\documentclass[11pt, oneside]{article}   	
\usepackage{geometry}                		
\usepackage{graphicx}				
\usepackage[utf8]{inputenc}
\usepackage[english]{babel}
\usepackage{cite}
\usepackage{fancyhdr}
\usepackage[affil-it]{authblk}

\usepackage{amssymb}
\usepackage{amsmath}
\usepackage{tikz-cd}
\usepackage{amsthm}
\newtheorem{theorem}{Theorem}[section]
\newtheorem*{theorem*}{Theorem}
\newtheorem{definition}{Definition}[section]
\newtheorem{corollary}{Corollary}[section]
\newtheorem{lemma}{Lemma}[section]
\newtheorem{proposition}{Proposition}[section]
\theoremstyle{definition}
\newtheorem{remark}{Remark}[section]

\newcommand{\C}{\mathbb C}
\newcommand{\R}{\mathbb R}
\newcommand{\N}{\mathbb N}
\newcommand{\calC}{\mathcal C}

\title{Iterative Methods for Globally Lipschitz Nonlinear Laplace Equations} 
\author{Jie Xu\thanks{Email Address: xujie@bu.edu}}
\affil{Boston University, Department of Mathematics and Statistics, Boston, Massachusetts, U.S.A.}
\date{}

\begin{document}
\maketitle
\noindent \begin{abstract} We introduce an iterative method to prove the existence and uniqueness of the complex-valued nonlinear elliptic PDE of the form $ -\Delta u + F(u) = f $ with Dirichlet or Neumann boundary conditions on a precompact domain $ \Omega \subset \mathbb{R}^{n}$, where $ F : \mathbb{C} \rightarrow \mathbb{C} $ is Lipschitz.  The same method gives a solution to 
$ - \Delta_{g} u + F(u) = f $ for these boundary conditions on a smooth, compact Riemannian manifold $ (M, g) $ with $ \mathcal{C}^{1} $ boundary, where $ - \Delta_{g} $ is the Laplace-Beltrami operator. We also apply parametrix methods to discuss an integral version of these PDEs.
\end{abstract}

\noindent \section{Introduction}
Let $ (M, g) $ be a smooth, compact Riemannian manifold with $ \mathcal{C}^{1} $ boundary $ \partial M $. We 
introduce an iterative method to solve the nonlinear elliptic PDE of the form
\begin{equation}\label{one}
-\Delta_{g} u + F(u) = f, \  u : M \rightarrow \mathbb{C},
\end{equation}
where $-\Delta_g$ is the Laplace-Beltrami operator on $M$.
Here $ F: \mathbb{C} \rightarrow \mathbb{C} $ is a globally Lipschitz function, and  $ f : \Omega \rightarrow \mathbb{C} $ is sufficiently regular. We can take either Dirichlet or Neumann boundary conditions. To introduce the iterative method, we first work on a domain, {\it i.e.}, an open, bounded subset $ \Omega \subset \mathbb{R}^{n} $ with $ \mathcal{C}^{1} $ boundary, and then extend the results to $ (M, g) $. \\
\\
\noindent We now state the general form of the main theorem. We will discuss various versions of this theorem with appropriate choices of Hilbert spaces $G, G'$ of functions in the paper.
\begin{theorem*} Let $ (M, g) $ be a smooth, compact Riemannian manifold with $ \mathcal{C}^{1} $ boundary $ \partial M $, and choose a symmetric boundary condition $B:\partial M\to \C$ for  $-\Delta$ acting on a Hilbert space of functions $G$. Let $ C $ be the optimal Poincar\'e constant on $ M $ for the operator $-\Delta$ 
acting on $ G $. Assume that $ F : \mathbb{C} \rightarrow \mathbb{C} $ is a globally Lipschitz function with Lipschitz constant $ C_{1} $ such that $ C_{1} C^{2} < 1 $. Then the Cauchy  problem for the nonlinear elliptic equation
\begin{equation*}
-\Delta u + F(u) = f \; \text{in} \; M, Bu = 0 \; \text{on} \; \partial M
\end{equation*}
 has a  weak solution $ u \in G $, provided that $ f$ is in the Hilbert space $G^{'} $ 
 associated to $G$. 
\end{theorem*}

The existence of solutions of (\ref{one}) is obtained by an iterative method: We first construct a sequence of solutions of linear Poisson equations associated with (\ref{one}), and  prove that the sequence converges strongly in some appropriate norm. We then prove that the limit of the sequence solves (\ref{one}). This argument relies heavily on the analysis of the optimal constant in the  Poincar\'e inequality on $ \Omega $. We 
also introduce a parametrix method in Section 4, which provides a solution to an integrated version of  (\ref{one}). 
\\
\\
Our approach differs from the usual method of linearizing the PDE and 
 then perturbing the linearized solution to a solution of the original nonlinear PDE.  This approach 
 involves some  restrictions, either to sufficiently small boundary conditions or a sufficiently small domain within 
 an appropriate energy space, so that the increase of inhomogeneity as well as nonlinearity can be controlled. 
 In contrast, we only assume a global Lipschitz condition for the nonlinear term.
 Since we work on standard $ L^{2} $-based Sobolev spaces,  we anticipate further applications of our method. \\

\noindent In Section 2, we prove (Theorem 2.1) the existence and uniqueness of solutions of (\ref{one}),
where $ M  = \Omega \subset \mathbb{R}^{n} $ is a bounded domain,  $ g $ is the standard Euclidean metric, and the boundary operator is $ Bu = u $ ({\it i.e.,} Dirichlet conditions). The Hilbert spaces $ G $ and $ G^{'} $ are standard Sobolev spaces $ H_{0}^{1}(\Omega) $ and $ L^{2}(\Omega) $, which are defined in this section. We also prover the existence and uniqueness of solutions of some variants of (\ref{one}). \\
\\
In Section 3, we solve the Dirichlet problem on compact manifolds $(M,g)$ with boundary (Theorem 3.4), with the boundary operators and the manifold analogues of $ G = H_{0}^{1}(M)$ and $ G^{'} =L^{2}(M)$.
We also treat the Neumann problem on $(M,g)$ in Theorem 3.7, where the boundary operator is $ Bu = g(\nabla u, N)$ (the vanishing of the normal derivative), and  the Hilbert spaces $ G =H_{\perp}^{1}(M) $ and $ G^{'} = L_{\perp}^{2}(M)$ are Sobolev spaces defined in this section. The Neumann case is significantly more difficult than the Dirichlet case. We also discuss the difficulties of solving the nonlinear equation on closed manifolds.\\
\\
In Section 4, we introduce a parametrix method for both the Euclidean and manifold cases. We show that the parametrix method yields a new integral expression for the solution of nonlinear elliptic PDEs of the type we consider, and we derive energy estimates. \\
\\
Our approach is inspired by the work of 
Hintz-Vasy \cite{Hintz}, \cite{HVasy}, who applied an iterative method to solve the small data quasilinear wave equation $ \Box_{g(u)} u = f + q(u, du) $ in asymptotic de Sitter space,
 where $ g(0) = g_{dS} $ is the de Sitter metric, for appropriate domain and solution space. \\
\\
Of course, many mathematicians have considered nonlinear elliptic equations using non-iterative methods. We just mention a few examples. Evans \cite{Evans} and Ekeland-T\'emam \cite{Temam} treat real solutions of fully nonlinear second order elliptic PDE with a convexity condition on the nonlinear terms. Fixed point methods, like the Schauder's/Larey-Schaefer's fixed point theorems, play a prominent role in solving semilinear and quasilinear elliptic PDEs, see for instance Evans \cite{Evans}, Gilbarg-Trudinger \cite{GT}. In addition, the inverse function theorems on Banach spaces is used, for example in \cite{Aubin}, \cite{Don}, to solve sesquilinear PDEs $ Pu + F(u) = f $, where $ P $ is an appropriate 
differential operator and $ F $ is a $ \mathcal{C}^{1} $ Fr\'echet differentiable map between Banach spaces. (These PDEs are typically solved for real valued functions, but we can obtain the local solvability of complex-valued  PDEs with the same  nonlinearities by fixed point and inverse function theorem methods.) There are other approaches to the
uniqueness of solutions of nonlinear PDEs \cite{Evans}, \cite{GT}, since  fixed point and other methods may not guarantee the uniqueness of the fixed point, except in rare situations where an appropriate Banach space map is a contraction. In \cite{Stahn}, the existence and uniqueness is proven for the solution of a quasilinear second order elliptic PDE with a real-valued Caratheodory nonlinear function with both coerciveness and growth conditions. Recently, Caffarelli {\it et al.} \cite{CGGKV} introduced powerful methods in nonlinear PDEs, such as the De Giorgi Theorem in elliptic and parabolic PDEs, the concentration-compactness rigidity method in dispersive equations, etc. \\
\\
\noindent Iterative methods have a long history as well, as they can be traced back to Newton's method to solve nonlinear problems. On nonlinear elliptic PDE, iterative methods have traditionally been applied to solving discretized versions of elliptic PDE using numerical approximation. For continuous methods, there is a large body of work, summarized in \cite{Evans}, applying iterative methods to real-valued nonlinear elliptic PDEs by assuming the existence of subsolutions and supersolutions. However, Evans states: "Warning: ...A plausible plan for constructively solving a quasilinear elliptic PDE would be to select some $ u_{0} $ and then iteratively solve the linear boundary-value problems...However, we CANNOT assert that $ \lbrace u_{k} \rbrace_{k = 0}^{\infty} $ then converges to a solution...Schauder's and Schaefer's fixed point theorems do not say that any sequence converges to a fixed point." In another direction, Nash \cite{Nash} modified Newton's method in his proof of the isometric embedding theorem. This work developed  into the so-called Nash-Moser method \cite{Moser1}, \cite{Moser2} to solve nonlinear equations iteratively, with the drawback that typically some regularity is lost at each iteration. Later H\"ormandar \cite{Hor1} improved Moser's method to reduce the loss of regularity, which is actually closer to Nash's original approach. 

\noindent \section{The Dirichlet Problem for Domains in Euclidean Space} 
In this section, we consider the Dirichlet problem for the basic equation (\ref{one})
on a connected bounded domain $\Omega\subset\R^n$ with $\calC^1$ boundary.  Thus we solve $-\Delta u + F(u) = f$, provided  $ f \in L^{2}(\Omega) $ and $ F : \mathbb{C} \rightarrow \mathbb{C} $ is globally Lipschitz. (An example of a non-differentiable globally Lipschitz function  is $ F(z) = e^{i\Im z} $, which has global Lipschitz constant $ \sqrt{2} $.)Here $ -\Delta $ is the standard Laplacian with strictly positive eigenvalues. The main existence and uniqueness result is Theorem 2.1. We also consider some variations of (\ref{one}) in this section.

\subsection{The Main Result}
\noindent In this subsection, we prove the existence and uniqueness of solutions of (\ref{one}).
\begin{definition} \label{Def 2.1} Let $ \Omega \subset \mathbb{R}^{n} $ be a connected bounded domain.
\\
(i) The $L^2$ inner product $ \langle \cdot, \cdot \rangle_{L^{2}} $ on $ \mathcal{C}_{c}^{\infty}(\Omega,\C) $ is
\begin{equation*}
\langle f, h \rangle_{L^{2}} = \int_{\Omega} f \bar{h} dx.
\end{equation*}
\noindent $ L^{2}(\Omega) $ is the completion of $ \mathcal{C}_{c}^{\infty} (\Omega,\C)$ with respect to $ \langle \cdot, \cdot \rangle_{L^{2}} $. The norm on $ L^{2}(\Omega) $ is thus given by $ \left\lVert f \right\rVert_{L^{2}(\Omega)}^{2} = \langle f, f \rangle_{L^{2}} $. \\
\\
(ii) The energy inner product $ \langle \cdot, \cdot \rangle_{E} $ on $ \mathcal{C}_{c}^{\infty}(\Omega,\C) $ is
\begin{equation*}
\langle f, h \rangle_{E} = \int_{M} \nabla f \cdot \nabla \bar h dx.
\end{equation*}
(iii) The $H^1$ inner product is  $ \langle \cdot, \cdot \rangle_{H^{1}} $ on $ \mathcal{C}^{\infty}(\Omega,
\C) $ or $ \mathcal{C}_{c}^{\infty}(\Omega,\C) $ is
\begin{equation*}
\langle f, h \rangle_{H^{1}} = \langle f, h \rangle_{L^{2}} + \langle f, h \rangle_{E}.
\end{equation*}
\noindent  $ H^{1}(\Omega) $ is the completion of $ \mathcal{C}^{\infty}(\Omega,\C) $ 
with respect to $ \langle \cdot, \cdot \rangle_{H^{1}} $,
 and $H^1_0(\Omega)  $ is the completion of $ \mathcal{C}^{\infty}_c(\Omega,\C) $ with respect to $ \langle \cdot, \cdot \rangle_{H^{1}} $. $H_0^1(\Omega)$ is the space of weakly differentiable functions on $\Omega$ that vanish on $\partial \Omega.$ \\
 \\
(iv) In general, the integer-valued Sobolev spaces $ H^{k}(\Omega) $, $ k > 1 $ is defined inductively by requiring that
\begin{equation*}
    u \in H^{k}(\Omega) \Leftrightarrow \frac{\partial u}{\partial x_{j}} \in H^{k-1}(\Omega), j = 1, \dotso, n.
\end{equation*}
\noindent Here $ x_{1}, \dotso, x_{n} $ are coordinate system on $ \Omega \subset \mathbb{R}^{n} $. The norm $ \lVert \cdot \rVert_{H^{k}(\Omega)} $ is given by
\begin{equation*}
    \lVert u \rVert_{H^{k}(\Omega)} = \lVert u \rVert_{H^{k-1}(\Omega)} + \sum_{j=1}^{n} \left\lVert \frac{\partial u}{\partial x_{j}} \right\rVert_{H^{k-1}(\Omega)}.
\end{equation*}
\end{definition}
\noindent We first prove a result about the  optimal Poincar\'e constant for $ u \in H_{0}^{1}(\Omega) $, or equivalently for $u$ with Dirichlet boundary condition.  We use the well-known fact that $-\Delta$ on $H^1_0(\Omega)$ has discrete spectrum contained in $\R^+$. See for example \cite[Ch. 28, 31, 32]{Lax}
\begin{proposition} \label{Prop 2.1} (Poincar\'e inequality) Let  $ \lambda_{1} $ be the smallest eigenvalue of $-\Delta$  on $H_0^1(\Omega) $. For  $ u \in H_{0}^{1}(\Omega) $, we have
\begin{equation*}\label{P.cons.}
\left\lvert u \right\rvert_{L^{2}(\Omega)} \leqslant \lambda_{1}^{-1 \slash 2} \left\lvert \nabla u \right\rvert_{L^{2}(\Omega)}.
\end{equation*}
Moreover, $\lambda_1^{-1/2}$ is the maximal constant for which this estimate holds.
\end{proposition}
\begin{proof} See \cite[Ch. 7]{GT}. The optimal constant is easily obtained by minimizing the Rayleigh quotient.
\end{proof}
\noindent Now we can prove our main result.  The proof introduces the iterative method 
that is used throughout the paper. 

\begin{theorem}\label{Thm 2.1} Let $ \Omega \subset \mathbb{R}^{n} $ be a connected bounded domain with $ \mathcal{C}^{1} $ boundary. Let $ C = \lambda_{1}^{-1\slash 2} $ be the optimal Poincar\'e constant on $ \Omega $, as in Proposition \ref{Prop 2.1}. Assume that $ F : \mathbb{C} \rightarrow \mathbb{C} $ is a globally Lipschitz function with Lipschitz constant $ C_{1} $ satisfying $ C_{1} C^{2} < 1 $, and choose $ f \in L^{2}(\Omega) $. Then the Dirichlet problem for the nonlinear elliptic equation
\begin{equation}\label{Thm 2.1 PDE}
-\Delta u + F(u) = f \; \text{in} \; \Omega, u |_{\partial \Omega} = 0,
\end{equation}
\noindent has a unique weak solution $ u \in H_{0}^{1}(\Omega) $.
\end{theorem}
\begin{proof} We may assume $ F(0) = 0 $, since if $ F(0) = c \neq 0 $, then we replace $F$ with 
 $ \tilde{F} = F - c $. It is clear that $ \tilde{F} $ is still globally Lipschitz.\\
\\
We first prove existence using our iterative set-up. We want to construct a sequence $ \lbrace u_{k} \rbrace_{k = 0}^{\infty} \subset H^2(\Omega)$ such that
\begin{align} 
 -\Delta u_{0} &= f \; in \; \Omega, u_{0} = 0 \; on \; \partial \Omega; \label{Thm. 2.1 strong-iteration1}\\
 -\Delta u_{k} + F(u_{k-1}) &= f \; in \; \Omega, u_{k} = 0 \; on \; \partial \Omega, k = 1, 2, 3, \dotso. \label{Thm. 2.1 strong-iteration2}
\end{align}
\noindent According to standard linear elliptic theory \cite{Evans2}, (\ref{Thm. 2.1 strong-iteration1}) has a unique solution with $ u_{0} \in H^{2}(\Omega) $; this uses the  Poincar\'e inequality, the Lax-Milgram theorem, and the regularity of the Poisson equation. In addition, we have $ u_{0} \in H_{0}^{1}(\Omega) $ in the trace class sense, so $ u_{0} \in H_{0}^{1}(\Omega) \cap H^{2}(\Omega) $.  \\
\\
Observe that for each $ k \geqslant 1 $ in (\ref{Thm. 2.1 strong-iteration2}), if we know that $ u_{k - 1} \in H^{2}(\Omega) \cap H_{0}^{1}(\Omega) $, then by the Lipschitz property of $ F $, we have
\begin{align}\label{Thm 2.1 Lip-esti}
\left\lVert F(u_{k-1}) \right\rVert_{L^{2}(\Omega)} & = \left\lVert F(u_{k-1}) - F(0) \right\rVert_{L^{2}(\Omega)} = \int_{\Omega} \left\lvert F(u_{k-1}) - F(0) \right\rvert^{2} dx \nonumber \\
& \leqslant \int_{\Omega} C_{1}^{2} \left\lvert u_{k-1} \right\rvert^{2} dx = C_{1}^{2} \left\lVert u_{k-1} \right\rVert_{L^{2}(\Omega)} \nonumber \\
& \leqslant C_{1}^{2} \left\lVert u_{k-1} \right\rVert_{H^{2}(\Omega)}.
\end{align}
\noindent Therefore $ u_{k-1} \in H^{2}(\Omega) $ implies $ F(u_{k-1}) \in L^{2}(\Omega) $.
Thus we can rewrite (4) as
\begin{equation}
-\Delta u_{k} = - F(u_{k-1}) + f \; in \; \Omega, u = 0 \; on \; \partial \Omega, 
\end{equation}
where the right hand side is in $L^2(\Omega).$
\noindent Standard linear elliptic theory implies that (6) has a unique solution $ u_{k} \in H_{0}^{1}(\Omega) \cap H^{2}(\Omega)$ for each $k$.\\
\\
We want to show that there exists $ u \in H_{0}^{1}(\Omega) $ such that we have the strong convergence
\begin{equation}
u = \lim_{k \rightarrow \infty} u_{k} \; \text{in} \: H_{0}^{1}(\Omega).
\end{equation}
To do this, we take consecutive iterations of  (4): 
\begin{align}
 -\Delta u_{k} + F(u_{k-1}) &= f \; in \; \Omega, u = 0 \; on \; \partial \Omega; \\
 -\Delta u_{k+1} + F(u_{k}) &= f \; in \; \Omega, u = 0 \; on \; \partial \Omega,
\end{align}
\noindent Subtracting (8) from (9),  we obtain
\begin{equation}
-\Delta (u_{k+1} - u_{k}) = -F(u_{k}) + F(u_{k-1}).
\end{equation}
\noindent Since $ u_{k} \in H_{0}^{1}(\Omega) \cap H^{2}(\Omega) $, we can introduce the weak version of (10):
\begin{equation}
\int_{\Omega} (\nabla u_{k+1} - \nabla u_{k}) \cdot \nabla v dx = - \int_{\Omega} (F(u_{k}) - F(u_{k-1})) \cdot v dx,
\end{equation}
for all $ v \in H_{0}^{1}(\Omega) $.
Setting $ v = u_{k+1} - u_{k} \in H_{0}^{1}(\Omega) $, we get:
\begin{equation*}
\int_{\Omega} \left\lvert \nabla u_{k+1} - \nabla u_{k} \right\rvert^{2} dx = - \int_{\Omega} (F(u_{k}) - F(u_{k-1})) \cdot (u_{k+1} - u_{k}) dx.
\end{equation*}
\noindent This gives
\begin{align*} \lefteqn{
 \left\lVert \nabla u_{k+1} - \nabla u_{k} \right\rVert_{L^{2}(\Omega)}^{2}}\\
 & = - \int_{\Omega} (F(u_{k}) - F(u_{k-1})) \cdot (u_{k+1} - u_{k}) dx 
\leqslant  \int_{\Omega} \left\lvert F(u_{k}) - F(u_{k-1}) \right\rvert \cdot \left\lvert u_{k+1} - u_{k} \right\rvert dx\\
&\leqslant \int_{\Omega} C_{1} \left\lvert u_{k} - u_{k-1} \right\rvert \cdot \left\lvert u_{k+1} - u_{k} \right\rvert dx \nonumber 
\leqslant  C_{1} \left\lVert u_{k} - u_{k-1} \right\rVert_{L^{2}(\Omega)} \cdot \left\lVert u_{k+1} - u_{k} \right\rVert_{L^{2}(\Omega)} \\
&\leqslant  C_{1} C^{2} \left\lVert \nabla u_{k} - \nabla u_{k-1} \right\rVert_{L^{2}(\Omega)} \cdot \left\lVert \nabla u_{k+1} - \nabla u_{k} \right\rVert_{L^{2}(\Omega)}. \nonumber 
\end{align*}
\noindent It follows that
\begin{equation}
\left\lVert \nabla u_{k+1} - \nabla u_{k} \right\rVert_{L^{2}(\Omega)} \leqslant \sqrt{C_1C_2^2} \left\lVert \nabla u_{k} - \nabla u_{k-1} \right\rVert_{L^{2}(\Omega)}, \forall k \in \mathbb{N}.
\end{equation}
\noindent Since $C_{1}C^{2}  < 1 $ by hypothesis, we conclude that $ \lbrace \nabla u_{k} \rbrace_{k=0}^{\infty} $ is Cauchy in $ L^{2}(\Omega) $. The Poincar\'e inequality guarantees the equivalence of $ H_{0}^{1} $-norm of $ u $ and $ L^{2} $-norm of $ \nabla  u $, and hence $ \lbrace u_{k} \rbrace_{k=0}^{\infty} $ is a Cauchy sequence in $ H_{0}^{1}(\Omega) $. Since $ H_{0}^{1}(\Omega) $ is a Hilbert space, we conclude that there exists a unique $ u \in H_{0}^{1}(\Omega) $ such that
\begin{equation}
u = \lim_{k \rightarrow \infty} u_{k}\ in\ H_{0}^{1}(\Omega).
\end{equation}
\noindent We now show that the limit in (13) solves (2) in the weak sense. The weak version of the iterative equation (4) for any $ v \in H_{0}^{1}(\Omega) $ is:
\begin{equation}\label{Thm2.1-weak}
\int_{\Omega} \nabla u_{k} \cdot \nabla v dx + \int_{\Omega} F(u_{k-1}) \cdot v dx = \int_{\Omega} f \cdot v dx.
\end{equation}
\noindent Note that since the limit $ u $ is in $ H_{0}^{1}(\Omega) $, the integrals in the left hand side of 
(14)  converge when $ u_{k}, u_{k-1} $ are replaced by $ u $. \\
\\
We can easily estimate the size of 
\begin{equation*}
\int_{\Omega} (\nabla u - \nabla u_{k}) \cdot \nabla v dx + \int_{\Omega} (F(u) - F(u_{k-1})) v dx : = A + B.
\end{equation*}
For $ A $, by Cauchy-Schwarz, we have
\begin{align*}
\left\lvert A \right\rvert & = \left\lvert \int_{\Omega} (\nabla u - \nabla u_{k}) \cdot \nabla v dx \right\rvert \leqslant \left\lVert \nabla u - \nabla u_{k} \right\rVert_{L^{2}(\Omega)} \cdot \left\lVert \nabla v \right\rVert_{L^{2}(\Omega)} \xrightarrow{k \rightarrow \infty} 0.
\end{align*}
\noindent Similarly, we get
\begin{align*}
\left\lvert B \right\rvert & = \left\lvert \int_{\Omega} (F(u) - F(u_{k-1})) v dx \right\rvert \leqslant \int_{\Omega} \left\lvert F(u) - F(u_{k-1}) \right\rvert \cdot \left\lvert v \right\rvert dx \\
& \leqslant C_{1} \int_{\Omega} \left\lvert u - u_{k-1} \right\rvert \cdot \left\lvert v \right\rvert dx \leqslant C_{1} \left\lVert u - u_{k-1} \right\rVert_{L^{2}(\Omega)} \cdot \left\lVert v \right\rVert_{L^{2}(\Omega)} \xrightarrow{j \rightarrow \infty} 0.
\end{align*}
\noindent We conclude using (\ref{Thm2.1-weak}) that
\begin{align*}
& \left\lvert \int_{\Omega} \nabla u \cdot \nabla v dx + \int_{\Omega} F(u) \cdot v dx - \int_{\Omega} f \cdot v dx \right\rvert \\
& \leqslant \left\lvert \int_{\Omega} (\nabla u - \nabla u_{k}) \cdot \nabla v dx + \int_{\Omega} (F(u) - F(u_{k-1})) v dx \right\rvert \\
& \qquad + \left\lvert \int_{\Omega} \nabla u_{k} \cdot \nabla v dx + \int_{\Omega} F(u_{k-1}) \cdot v dx - \int_{\Omega} f \cdot v dx \right\rvert \\
& = \left\lvert \int_{\Omega} (\nabla u - \nabla u_{k}) \cdot \nabla v dx + \int_{\Omega} (F(u) - F(u_{k-1})) v dx \right\rvert.
\end{align*}
\noindent  
The right hand side of the last inequality is arbitrarily small for $k \gg 0$, so
\begin{equation}
\int_{\Omega} \nabla u \cdot \nabla v dx + \int_{\Omega} F(u) \cdot v dx = \int_{\Omega} f \cdot v dx, \forall v \in H_{0}^{1}(\Omega).
\end{equation}
\noindent Thus $u$ solves (\ref{Thm 2.1 PDE}) in the weak sense.\\
\\
To check the boundary condition, since each $ u_{k} $ solves (4) in the weak sense, we obtain by the trace theorem that $ u_{k} |_{\partial \Omega} = Tu_{k} = 0 $ where $ T: H_{0}^{1}(\Omega) \rightarrow L^{2}(\partial \Omega), Tv = v|_{\partial\Omega} $ is the trace operator. By the continuity of $ T $, it follows that $ Tu = 0 $.\\
\\
For the uniqueness of the weak solution, suppose that $ u_{1}, u_{2} $ solve (2) in the weak sense:
\begin{align*}
 \int_{\Omega} \nabla u_{1} \cdot \nabla v dx + \int_{\Omega} F(u_{1}) v dx &= \int_{\Omega} fvdx, \\
 \int_{\Omega} \nabla u_{2} \cdot \nabla v dx + \int_{\Omega} F(u_{2}) v dx &= \int_{\Omega} fvdx.
\end{align*}
\noindent Subtracting, we have
\begin{equation}
\int_{\Omega} (\nabla u_{1} - \nabla u_{2}) \cdot \nabla v dx = \int_{\Omega} (F(u_{2}) - F(u_{1})) v dx.
\end{equation}
\noindent Setting $ v = u_{1} - u_{2} $, we get
\begin{align*}
& \left\lVert \nabla u_{1} - \nabla u_{2} \right\rVert_{L^{2}(\Omega)}^{2} \leqslant \int_{\Omega} \left\lvert F(u_{2}) - F(u_{1}) \right\rvert \left\lvert u_{1} - u_{2} \right\rvert dx \\
 \leqslant & C_{1} \left\lVert u_{1} - u_{2} \right\rVert_{L^{2}(\Omega)}^{2} \leqslant C_{1} C^{2} \left\lVert \nabla u_{1} - \nabla u_{2} \right\rVert_{L^{2}(\Omega)}^{2}.
\end{align*}
\noindent Since $ C_{1}C^{2} < 1 $, we obtain $ \left\lvert \nabla u_{1} - \nabla u_{2} \right\rvert_{L^{2}(\Omega)} = 0 $. It follows that $u_{1} - u_{2} = c,$ for some constant $ c \in \mathbb{C} $. Since $ u_{1} - u_{2} = 0 $ on $ \partial \Omega $, we must have $c=0$, {\it i.e.,} $u_1 = u_2.$
\end{proof}
\begin{remark} 
In this proof, we have used both strong and weak versions of linear and nonlinear elliptic PDEs. We use the strong version as long as possible to show the iterative steps clearly, and then switch to weak version out of necessity. In fact, the iterative steps  (\ref{Thm. 2.1 strong-iteration1}) - (10) work equally well in the weak form, since we can rewrite (\ref{Thm. 2.1 strong-iteration1}), (\ref{Thm. 2.1 strong-iteration2}) for some $ v \in H_{0}^{1}(\Omega) $ as
\begin{align*}
 \int_{\Omega} \nabla u_{0} \cdot \nabla v dx &= \int_{\Omega} fv dx, \\
 \int_{\Omega} \nabla u_{k} \cdot \nabla v dx + \int_{\Omega} F(u_{k-1}) \cdot v dx &= \int_{\Omega} f v dx, k \geqslant 1,
\end{align*}
\noindent and then continue.
\end{remark}

\subsection{More General Nonlinearities and More General Dirichlet Problems}
\noindent In this subsection, we  discuss the iterative method on Euclidean space when $ F $ is a function of both $ u $ and $ \nabla u $. \\
\\
The first result is the existence of a solution of $ -\Delta u + F(u, \nabla u) = f $ when $ F $ is of the simple form $ F(u, \nabla u) = F_{1}(u) + F_{2}(\nabla u) $, where $ F_{1} $ and $ F_{2} $ are globally Lipschitz functions with appropriate Lipschitz constants.
\begin{theorem} Let $ \Omega \subset \mathbb{R}^{n} $ be a connected domain with $ \mathcal{C}^{1} $ boundary. Let $ C = \lambda_{1}^{-1 \slash 2} $ be the optimal Poincar\'e constant on $ \Omega $, as in Proposition \ref{Prop 2.1}. Assume that $ F_{1} : \mathbb{C} \rightarrow \mathbb{C} $, $ F_{2} : \mathbb{C}^{n} \rightarrow \mathbb{C} $ are globally Lipschitz functions with Lipschitz constants $ C_{1}, C_{2} $ such that $ C_{1} C^{2} + C_{2}C < 1 $. Choose $f\in L^2(\Omega).$ Then the Dirichlet problem for the nonlinear elliptic equation
\begin{equation}\label{Thm 2.2 equation}
-\Delta u + F_{1}(u) + F_{2}(\nabla u) = f \; \text{in} \; \Omega, u |_{\partial \Omega} = 0,
\end{equation}
\noindent has a unique weak solution $ u \in H_{0}^{1}(\Omega) $.
\end{theorem}
\begin{proof} We apply the iterative method as in Theorem \ref{Thm 2.1}. As in that proof, we can assume that $ F_{1}(0) = F_{2}(0) = 0 $. Construct $ \lbrace u_{k} \rbrace, k\in \mathbb N, $ such that
\begin{align}
 -\Delta u_{0} = f, u_{0} |_{\partial \Omega} &= 0; \\
 - \Delta u_{k} + F_{1}(u_{k-1}) + F_{2}(\nabla u_{k-1}) = f, u_{k} |_{\partial \Omega} &= 0.\label{Thm 2.2 iteration-strong}
 \end{align}
\noindent Applying (\ref{Thm 2.1 Lip-esti}) to $ F_{1}(u_k) $ and $ F_{2} (\nabla u_k)$, we conclude that
\begin{equation*}
\left\lVert F_{1}(u_{k}) \right\rVert_{L^{2}(\Omega)} \leqslant \tilde{C}_{1} \left\lVert u_{k} \right\rVert_{H^{2}(\Omega)}, \left\lVert F_{2}(\nabla u_{k}) \right\rVert_{L^{2}(\Omega)} \leqslant \tilde{C}_{2} \left\lVert u_{k} \right\rVert_{H^{2}(\Omega)}.
\end{equation*}
\noindent for some constants $ \tilde{C}_{1}, \tilde{C}_{2} > 0 $, provided $ u_{k} \in H^{2}(\Omega) $.  As before, it follows that $ u_{k} \in H_{0}^{1}(\Omega) \cap H^{2}(\Omega) $. Subtracting two consecutive iterative equations in the weak form of (\ref{Thm 2.2 iteration-strong}), we have
\begin{equation*}
\int_{\Omega} (\nabla u_{k+1} - \nabla u_{k}) \cdot \nabla v dx = -\int_{\Omega} (F_{1}(u_{k}) - F_{1}(u_{k-1})) \cdot v dx - \int_{\Omega} (F_{2}(\nabla u_{k}) - F_{2}(\nabla u_{k-1})) \cdot v dx,
\end{equation*}
\noindent for any $ v \in H_{0}^{1}(\Omega) $. Choosing $ v = u_{k+1} - u_{k} $, we compute that
\begin{align*}  \lefteqn{
 \left\lVert \nabla u_{k+1} - \nabla u_{k} \right\rVert_{L^{2}(\Omega)}^{2} }\\
 & \leqslant \int_{\Omega} \left\lvert F_{1}(u_{k}) - F_{1}(u_{k-1}) \right\rvert v dx + \int_{\Omega} \left\lvert F_{2}(\nabla u_{k}) - F_{2}(\nabla u_{k-1}) \right\rvert v dx \\
& \leqslant  \int_{\Omega} C_{1} \left\lvert u_{k} - u_{k-1} \right\rvert \cdot \left\lvert u_{k+1} - u_{k} \right\rvert dx + \int_{\Omega} C_{2} \left\lvert \nabla u_{k} - \nabla u_{k-1} \right\rvert \cdot \left\lvert u_{k+1} - u_{k} \right\rvert dx \\
&\leqslant  C_{1} \left\lVert u_{k} - u_{k-1} \right\rVert_{L^{2}(\Omega)} \left\lVert u_{k+1} - u_{k} \right\rVert_{L^{2}(\Omega)} + C_{2} \left\lVert \nabla u_{k} - \nabla u_{k-1} \right\rVert_{L^{2}(\Omega)} \cdot \left\lVert u_{k+1} - u_{k} \right\rVert_{L^{2}(\Omega)} \\
&\leqslant  (C_{1} C^{2} + C_{2} C) \left\lVert \nabla u_{k} - \nabla u_{k-1} \right\rVert_{L^{2}(\Omega)} \cdot \left\lVert \nabla u_{k+1} - \nabla u_{k} \right\rVert_{L^{2}(\Omega)}.
\end{align*}
\noindent Thus
\begin{equation}
\left\lVert \nabla u_{k+1} - \nabla u_{k} \right\rVert_{L^{2}(\Omega)} \leqslant  \sqrt{C_{1} C^{2} + C_{2} C}
 \left\lVert \nabla u_{k} - \nabla u_{k-1} \right\rVert_{L^{2}(\Omega)}. 
\end{equation}
\noindent Since $C_{1} C^{2} + C_{2} C<1$,  $\lbrace \nabla u_{k} \rbrace $ is a Cauchy sequence in $ L^{2}(\Omega) $.
By the Poincar\'e inequality, $ \lbrace u_{k} \rbrace $ is Cauchy in $ H_{0}^{1}(\Omega) $. Therefore $ \lbrace u_{k} \rbrace $ converges to a limit $ u \in H_{0}^{1}(\Omega) $.\\
\\
We prove that $ u $ solves (\ref{Thm 2.2 equation}) in the weak sense. The integrals in the weak version of (\ref{Thm 2.2 iteration-strong}),
\begin{equation*}
\int_{\Omega} \nabla u_{k} \cdot \nabla v dx + \int_{\Omega} (F_{1}(u_{k-1})v + F_{2}(\nabla u_{k-1}) v)dx = \int_{\Omega} fvdx,
\end{equation*}
\noindent  still converge when we replace $ u_{k}, u_{k-1} $ by $ u $. We now estimate
\begin{equation*}
\int_{\Omega} (\nabla u - \nabla u_{k}) \cdot v + \int_{\Omega} (F_{1}(u) - F_{1}(u_{k-1})v dx + \int_{\Omega} (F_{2}(\nabla u) - F_{2}(\nabla u_{k-1})) vdx : = A + B + C.
\end{equation*}
\noindent Repeating the argument in Theorem \ref{Thm 2.1}, we obtain $\left\lvert A \right\rvert \xrightarrow{k \rightarrow \infty} 0, \left\lvert B \right\rvert \xrightarrow{k \rightarrow \infty} 0.$ For $ C $, we compute that
\begin{align*}
& \left\lvert C \right\rvert = \left\lvert \int_{\Omega} (F_{2}(\nabla u) - F_{2}(\nabla u_{k-1})) vdx \right\rvert \leqslant \int_{\Omega} \left\lvert F_{2}(\nabla u) - F_{2}(\nabla u_{k-1}) \right\rvert \left\lvert v \right\rvert dx \\
\leqslant & C_{2} \int_{\Omega} \left\lvert \nabla u - \nabla u_{k-1} \right\rvert \left\lvert v \right\rvert dx \leqslant C_{2} \left\lVert \nabla u - \nabla u_{k-1} \right\rVert_{L^{2}(\Omega)} \cdot \left\lVert v \right\rVert_{L^{2}(\Omega)} \\
\leqslant & C_{2} \left\lVert u - u_{k-1} \right\rVert_{H_{0}^{1}(\Omega)} \cdot \left\lVert v \right\rVert_{L^{2}(\Omega)} \xrightarrow{k \rightarrow \infty} 0.
\end{align*}
\noindent We observe that
\begin{align*}
& \left\lvert \int_{\Omega} \nabla u \cdot \nabla v dx + \int_{\Omega} (F_{1}(u) + F_{2}(\nabla u)) vdx - \int_{\Omega} fv dx \right\rvert \\
\leqslant & \left\lvert \int_{\Omega} (\nabla u - \nabla u_{k}) \cdot v + \int_{\Omega} (F_{1}(u) - F_{1}(u_{k-1})v dx + \int_{\Omega} (F_{2}(\nabla u) - F_{2}(\nabla u_{k-1})) vdx \right\rvert  \\
& \qquad + \left\lvert \int_{\Omega} \nabla u_{k} \cdot \nabla v dx + \int_{\Omega} (F_{1}(u_{k-1})v + F_{2}(\nabla u_{k-1}) v)dx - \int_{\Omega} fv dx \right\rvert \\
= &  \left\lvert \int_{\Omega} (\nabla u - \nabla u_{k}) \cdot v + \int_{\Omega} (F_{1}(u) - F_{1}(u_{k-1})v dx + \int_{\Omega} (F_{2}(\nabla u) - F_{2}(\nabla u_{k-1})) vdx \right\rvert.
\end{align*}
\noindent 
The second to last line above vanishes due to (\ref{Thm 2.2 iteration-strong}). By the estimates for $ A, B, C $, the last line is arbitrarily small for $ k \gg 0$. It follows that $ u $ solves (17) in the weak sense. \\
\\
As before, since each $ u_{k} $ solves (\ref{Thm 2.2 iteration-strong}), we obtain by the continuity of the trace operator that $ u_{k} |_{\partial \Omega} = Tu_{k} = 0 $, which implies that $ Tu = 0 $.\\ 
\\
The uniqueness proof is essentially the same as in Theorem \ref{Thm 2.1}.
\end{proof} 

\noindent We now discuss the PDE of the form $ -\Delta u + F_{1}(u) + F_{2}(\nabla u) = f, u |_{\partial \Omega} = g $ with general Dirichlet condition, which has non-zero trace class. Suppose that $ f \in L^{2}(\Omega) $ and $ g \in L^{2}(\partial \Omega) $ such that $ g $ is the trace of some $ \tilde{g} \in H^{1}(\Omega) $. In this case, we set $ \tilde{u} = u - \tilde{g} $ and it follows that the above PDE becomes
\begin{equation*}
-\Delta \tilde{u} + F_{1}(\tilde{u} + \tilde{g}) + F_{2}(\nabla u + \nabla \tilde{g})= f - \Delta \tilde{g}, \; \text{in} \; \Omega, \tilde{u} = 0 \; \text{on} \; \partial \Omega.
\end{equation*}
\noindent Note that above expression does make sense in the weak form. To achieve this, we can rewrite the strong forms in (\ref{Thm. 2.1 strong-iteration1}) and (\ref{Thm. 2.1 strong-iteration2}) to weak form. We show the existence and uniqueness of the solution of PDE with general Dirichlet condition in the following corollary.
\begin{corollary}\label{Cor 2.1}
Let $ \Omega \subset \mathbb{R}^{n} $ be an open, bounded, connected subset of n-dimensional Euclidean space with $ \mathcal{C}^{1} $ boundary and symmetric boundary condition. Let $ C = \lambda_{1}^{-1 \slash 2} $ be the optimal Poincar\'e constant on $ \Omega $, as in Proposition \ref{Prop 2.1}. Assume that $ F_{1} : \mathbb{C} \rightarrow \mathbb{C} $, $ F_{2} : \mathbb{C}^{n} \rightarrow \mathbb{C} $ are globally Lipschitz functions with Lipschitz constants $ C_{1}, C_{2} $ such that $ C_{1} C^{2} + C_{2}C < 1 $. Then the Cauchy problem of nonlinear elliptic equation
\begin{equation}\label{Cor 2.1 equation}
-\Delta u + F_{1}(u) + F_{2}(\nabla u) = f \; \text{in} \; \Omega, u |_{\partial \Omega} = g,
\end{equation}
\noindent has a unique weak solution $ u \in H_{0}^{1}(\Omega) $, provided that $ f \in L^{2}(\Omega) $, $ g \in L^{2}(\partial \Omega) $ and $ g = \tilde{g} $ on $ \partial \Omega $ in the trace sense where $ \tilde{g} \in H^{1}(\Omega) $.
\end{corollary}
\begin{proof} We set $ \tilde{u} = u - \tilde{g} $, then (\ref{Cor 2.1 equation}) is equivalent to
\begin{equation}\label{Cor 2.1 equation2}
-\Delta \tilde{u} + F_{1}(\tilde{u} + \tilde{g}) + F_{2}(\nabla \tilde{u} + \nabla \tilde{g}) = f - \Delta \tilde{g}, \tilde{u} = 0 \; \text{on} \; \partial \Omega,
\end{equation}
\noindent in the weak sense. Note that the iteration equations become
\begin{equation*}
-\Delta \tilde{u}_{k} + F_{1}(\tilde{u}_{k-1} + \tilde{g}) + F_{2}(\nabla \tilde{u}_{k-1} + \nabla \tilde{g}) = f - \Delta \tilde{g}, \tilde{u}_{k} = 0 \; \text{on} \; \partial \Omega.
\end{equation*}
\noindent Where we have
\begin{equation*}
\left\lVert F_{1}(\tilde{u}_{k}) \right\rVert_{L^{2}(\Omega)} \leqslant \tilde{C}_{1} (\left\lVert \tilde{u}_{k} \right\rVert_{H^{2}(\Omega)} + \left\lVert \tilde{g} \right\rVert_{H^{1}(\Omega)}), \left\lVert F_{2}(\nabla \tilde{u}_{k}) \right\rVert_{L^{2}(\Omega)} \leqslant \tilde{C}_{2} (\left\lVert \tilde{u}_{k} \right\rVert_{H^{2}(\Omega)} + \left\lVert \tilde{g} \right\rVert_{H^{1}(\Omega)}).
\end{equation*} 
\noindent Then the solvability of (\ref{Cor 2.1 equation2}) as well as uniqueness of solution can be derived by exactly the same method as in Theorem 2.2.
\end{proof}
\begin{remark} We can apply this theory to find special solutions to the Schr\"odinger and wave equations on $ \Omega$ for functions $F(z) = f(|z|)z$ for a globally Lipschitz function $f:[0,\infty)\to \C.$ Note that $F(e^{i\xi t} z) = e^{i\xi t} F(z).$  If we assume that the solution of the Schr\"odinger equation
\begin{equation}\label{Schrodinger}
i\partial_{t}u + \Delta u = F(u), u|_{t = 0, x \in \partial \Omega} = g(x), u : [0, \infty) \times \Omega \rightarrow \mathbb{C},
\end{equation}
\noindent is of the form $ u(t, x) = e^{i\xi t} Q(x) $, then (\ref{Schrodinger}) reduces to
\begin{equation}
-\Delta Q(x) + \xi Q(x) = -F(Q(x)), Q(x) = g(x).
\end{equation}
\noindent We may apply Corollary \ref{Cor 2.1} to solve this generalized Dirichlet problem 
provided  $ \xi \geqslant 0 $. \\
\\
For the same choices of $F$ and $u(t,x)$, we can reduce the nonlinear wave equation
\begin{equation}
\partial_{tt}u + \Delta u = F(u), u|_{t = 0, x \in \partial \Omega} = g(x), u : [0, \infty) \times \Omega \rightarrow \mathbb{C},
\end{equation}
\noindent to the nonlinear elliptic equation
\begin{equation}\label{wave}
- \Delta Q(x) + \xi^{2} Q(x) = -F(Q(x)), Q |_{\partial \Omega} = g(x).
\end{equation}
\noindent We can again solve (\ref{wave}) by Corollary \ref{Cor 2.1}.
\end{remark}

\noindent \section{Nonlinear Laplace Equations on Compact Riemannian Manifolds with Boundary}
\noindent In this section, we consider the existence of solutions of (\ref{one}) on compact manifolds with boundary, for both Dirichlet and Neumann conditions. 
Throughout this section, $(M,\partial M, g)$ denotes a smooth, connected, compact Riemannian manifold $ (M, g) $ with $ \mathcal{C}^{1} $ boundary $ \partial M $.
We rely heavily on appropriate versions of the Poincar\'e inequality.
\subsection{Dirichlet Boundary Conditions}
\noindent In this subsection, we consider the existence of the solution of complex-valued nonlinear Poisson equation
\begin{equation}\label{Poisson manifold}
-\Delta_{g} u + F(u) = f, u \equiv 0 \; on \; \partial M,
\end{equation}
\noindent on $(M, \partial M, g).$ Here $-\Delta_g$ is the Laplace-Beltrami operator on $M$, and
$ F $ is a complex-valued globally Lipschitz function with Lipschitz constant $ C_{1}$. Our convention is that $ - \Delta_{g} $ has strictly positive eigenvalues. \\
\\
In order to solve (\ref{Poisson manifold}), we need a Poincar\'e inequality on $ (M, \partial M, g) $ with boundary; we then recall existence and regularity results for the linear Poisson
equation on $ (M, \partial M, g) $; finally, we apply a manifold version of the Euclidean iterative method.\\
\\
Let $ d\mu $ be the volume form on $ (M,\partial M, g) $. Using $d\mu$, we define the manifold analogues of the inner products in Definition \ref{Def 2.1}. There are several equivalent definitions of Sobolev spaces on manifolds, and we choose the one that is most convenient for our purposes.
\begin{definition} \label{Def 3.1} Let $ (M,\partial M,g) $ and $d\mu$ be the same as above. \\
(i) The $L^2$ inner product $ \langle \cdot, \cdot \rangle_{L^{2}} $ on $ \mathcal{C}^{\infty}(M,\C) $ is
\begin{equation*}
\langle f, h \rangle_{L^{2}} = \int_{M} f \bar{h} d\mu, \forall f.
\end{equation*}
\noindent  $ L^{2}(M) $ is the completion of $ \mathcal{C}^{\infty} (M,\C)$ with respect to $ \langle \cdot, \cdot \rangle_{L^{2}} $. The norm on $ L^{2}(M) $ is given by $ \left\lVert f \right\rVert_{L^{2}(M)}^{2} = \langle f, f \rangle_{L^{2}} $. \\
\\
(ii) The energy inner product $ \langle \cdot, \cdot \rangle_{E} $ on $ \mathcal{C}^{\infty}(M,\C) $ is
\begin{equation*}
\langle f, h \rangle_{E} = \int_{M} g(\nabla f, \nabla \bar{h}) d\mu.
\end{equation*}
Here $\nabla f$ is the gradient of $f$, which is given in local coordinates by
$\nabla f = g^{ij}\partial_jf (\partial/\partial x^i).$ \\
\\
(iii) The $H^1$ inner product $ \langle \cdot, \cdot \rangle_{H^{1}} $ on $ \mathcal{C}^{\infty}(M,\C) $ is
\begin{equation*}
\langle f, h \rangle_{H^{1}} = \langle f, h \rangle_{L^{2}} + \langle f, h \rangle_{E}, \forall f, h \in \mathcal{C}^{\infty}(M, \mathbb{C}).
\end{equation*}
\noindent $ H^{1}(M) $ is the completion of $ \mathcal{C}^{\infty}(M,\C) $ with respect to $ \langle \cdot, \cdot \rangle_{H^{1}} $.  $H^1_0(M)  $ is the completion of $ \mathcal{C}_{c}^{\infty}(\mathring M,\C) $ with respect to $ \langle \cdot, \cdot \rangle_{H^{1}} $. $H_0^1(M)$ is the space of weakly differentiable functions on $M$ that vanish on $\partial M.$ The norm on $ H^{1}(M) $ is thug given by $ \lVert f \rVert_{H^{1}(M)} = \langle f, f \rangle_{H^{1}} $. \\
\\
(iv) In general, the integer-ordered Sobolev spaces $ H^{k}(M) $, $ k > 1 $ is defined inductively by requiring that
 \begin{equation*}
     u \in H^{k}(M) \Leftrightarrow  \frac{\partial u}{\partial x_{j}} \in H^{k-1}(M), j = 1, \dotso, n.
 \end{equation*}
 \noindent Here $ x_{1}, \dotso, x_{n} $ are local coordinates on charts of $ M $.
\end{definition}
\noindent We now state two fundamental results about the spectral theory of $-\Delta_g$ on  $ (M, g) $ with a $ \mathcal{C}^{1} $ boundary  and the corresponding Poincar\'e inequality.
\begin{theorem} \cite[Thm. 4.4]{Aubin} For  $ (M, \partial M,g) $,  the eigenvalues of the Dirichlet problem for the Laplacian $ -\Delta_{g} $ on $ (M, g) $ are strictly positive. The eigenfunctions are in $ \mathcal{C}^{\infty}(M) $. The first eigenvalue $ \lambda_{1} $ equals 
 $$\text{inf}_{\phi\in \mathcal A} \left\lVert \phi \right\rVert_{E}^{2}, \ {\rm for}\   \mathcal{A} = \lbrace \phi \in H_{0}^{1}(M), \left\lVert \phi \right\rVert_{L^{2}(M)} = 1 \rbrace.$$
\end{theorem}
\begin{theorem}\label{Thm 3.2} \cite[Cor. 4.6]{Aubin}, {\it Poincar\'e Inequality.} For $ (M, \partial M,g) $ and $\lambda_1$ as above we have
\begin{equation}\label{ten}
\left\lVert \phi \right\rVert_{L^{2}(M)}^{2} \leqslant \lambda_{1}^{-1} \left\lVert \phi \right\rVert_{E}^{2}, \forall \phi \in H_{0}^{1}(M).
\end{equation}
\noindent Thus $ \left\lVert \phi \right\rVert_{E} $ is an equivalent norm for $ H_{0}^{1}(M) $.
\end{theorem}
\noindent We now consider the Dirichlet problem for the linear Poisson equation on $M$:
\begin{equation}\label{eleven}
-\Delta_{g} u = f, u \bigg|_{\partial M} = 0.
\end{equation}
\noindent Note that for $u, v \in C^\infty_c(\mathring M)$, 
\begin{equation}\label{motivation}\langle -\Delta_g u,v\rangle_{L^2} = \langle \nabla u, \nabla v
\rangle_{L^2} = \langle u,v\rangle_E.
\end{equation} The integration by parts in the first equality extends to functions with Dirichlet boundary conditions. This motivates the following definition.

\begin{definition} For $u, v \in H_{0}^{1}(M)$, set $B[u, v]  = \langle u, v \rangle_{E}.$
Choose $ f \in L^{2}(M) $. We say that $ u $ is a complex-valued weak solution of the Dirichlet boundary problem (\ref{eleven}) if
\begin{equation*}
B[u, v] = \langle f, v \rangle_{L^{2}},
\end{equation*}
\noindent for all $ v.$
\end{definition} 
\noindent The following theorem states the existence, uniqueness and regularity of weak solutions of (\ref{eleven}) on $ (M, \partial M, g) $. 
\begin{theorem}  \cite[Thm. 4.8]{Aubin} For $ (M, \partial M, g) $, there exists a unique weak solution $ u \in H_{0}^{1}(M) $ of (\ref{eleven}). If $ f \in \mathcal{C}^{\infty}(M) $, then $ u \in \mathcal{C}^{\infty}(M) $ and $u$ vanishes on the boundary. 
\end{theorem}
\noindent We now define the weak form of (\ref{Poisson manifold}) and prove the existence of weak solutions.
\begin{definition} Fix $ f \in L^{2}(M) $. For  $ u,v \in H_{0}^{1}(M) $, set
$$B_{F}[u, v] : = \langle u, v \rangle_{E} + \langle F(u),v\rangle_{L^2}.$$
$u$ is a complex-valued weak solution of the Dirichlet boundary problem (\ref{Poisson manifold}) if
\begin{equation*}
B_{F}[u, v] = \langle f, v \rangle_{L^{2}}.
\end{equation*}
\noindent for all $ v \in H_{0}^{1}(M) $.
\end{definition} 
\begin{theorem} \label{Thm 3.4} For $ (M,  \partial M, g) $, let $ C = \lambda_{1}^{-1 \slash 2} $ be the Poincar\'e constant on $ M $ as in Theorem 3.2. Let $ F : \mathbb{C} \rightarrow \mathbb{C} $ be a globally Lipschitz function with Lipschitz constant $ C_{1} $, and fix $f\in L^2(M)$. If  $ C_{1}C^{2} < 1 $, then (\ref{Poisson manifold}) has a weak solution $ u \in H_{0}^{1}(M) $.
\end{theorem}
\begin{proof}  The argument in Theorem \ref{Thm 2.1}, now using Theorem 3.2 and Theorem 3.3, implies that 
\begin{equation*}
\left\lVert F(u) \right\rVert_{L^{2}(M)} \leqslant C_{1} \left\lVert u \right\rVert_{H_{0}^{1}(M)}.
\end{equation*}
\noindent Choosing a test function $ v \in H_{0}^{1}(M) $, we find $\{u_k\}$ satisfying the weak iterative equations
\begin{align}\label{twelve}
 \langle u_{0}, v \rangle_{E} &= \langle f, v \rangle_{L^{2}}, \nonumber \\
 \langle u_{k}, v \rangle_{E} + \langle F(u_{k-1}), v \rangle_{L^{2}} &= \langle f, v \rangle_{L^{2}}.
\end{align}
\noindent As in Theorem \ref{Thm 2.1}, we conclude that there exists $ u = \lim_{k\to\infty} u_k \in H_{0}^{1}(M) $ satisfying
\begin{equation*}
B_{F}[u, v] = \langle f, v \rangle_{L^{2}}.
\end{equation*}
\end{proof}
\noindent As in the Euclidean case, we may replace the trivial Dirichlet boundary condition with the general Dirichlet condition. 
\begin{corollary} Under the assumptions of Theorem \ref{Thm 3.4},  the generalized Dirichlet problem 
\begin{equation}\label{31}
-\Delta_{g} u + F(u) = f \; \text{in} \; M, u |_{\partial M} = h,
\end{equation}
\noindent has a weak solution $ u \in H_{0}^{1}(M) $, provided $ h \in L^{2}(\partial M) $ is the trace of $\tilde h\in H^1(M).$
\end{corollary}
\begin{proof} For $ \tilde{u} = u - \tilde{h} $, (\ref{31}) is equivalent to
\begin{equation*}
-\Delta_{g} \tilde{u} + F(\tilde{u} + \tilde{h}) = f + \Delta_{g} \tilde{h}, \tilde{u} |_{\partial M} = 0.
\end{equation*}
\noindent We can now apply Theorem \ref{Thm 3.4}.
\end{proof}


\noindent \subsection{Neumann Boundary Condition}
In this subsection, we consider the existence and uniqueness of solutions to the  complex-valued Poisson equation
\begin{equation}\label{seventeen}
-\Delta_{g} u + F(u) = f, g(\nabla u, N) \equiv 0 \; on \; \partial M,
\end{equation}
\noindent on $(M, \partial M, g).$ Here $ N $ is the outward unit normal vector field along $ \partial M $. The main tools will be a Poincar\'e inequality and the Lax-Milgram theorem.\\
\\
We adapt Definition \ref{Def 3.1} to this situation.
\begin{definition} For $ (M ,\partial M, g) $ and $ d\mu $ as in the previous subsection, \\
(i) Set 
$$ \calC^{\perp} = \left\{u\in \calC^\infty(M): \int_{M} u d\mu = 0\right\}.$$
(ii) Set $ L_{\perp}^{2}(M) $ to be the completion of $ \calC^{\perp} $ with respect to the $L^2$ inner product on $M$, and let $ \lVert \cdot \rVert_{L_{\perp}^{2}(M)} $ be the corresponding norm. \\
\\
(iii) Set $ E $ to be the completion of $ \calC^{\perp} $ with respect to $ \langle \cdot, \cdot \rangle_{E} $, and let $ \left\lVert \cdot \right\rVert_{E} $ be the corresponding norm.\\
\\
(iv) The $ H_{\perp}^{1} $ inner product $ \langle \cdot, \cdot \rangle_{H_{\perp}^{1}} $ on $ \calC^{\perp} $ is
\begin{equation*}
\langle f, g \rangle_{H_{\perp}^{1}} = \langle f, g \rangle_{L_{\perp}^{2}} + \langle f, g \rangle_{E}, \forall f, g \in \calC^{\perp}.
\end{equation*}
\noindent Set $ H_{\perp}^{1}(M) $ to be the completion of $ \calC^{\perp} $ with respect to $ \langle \cdot, \cdot \rangle_{H_{\perp}^{1}} $, and let $ \lVert \cdot \rVert_{H_{\perp}^{1}(M)} $ be the corresponding norm.
\end{definition} 
\noindent Clearly, $ \int_{M} u = 0$ for $u \in L_{\perp}^{2}(M) $ or $ u \in H_{\perp}^{1}(M) $,  $ H_{\perp}^{1}(M) \subset L_{\perp}^{2}(M) $, and $\lVert u \rVert_{E} \leq \lVert u \rVert_{H_\perp^{1}(M)}$.\\
\\
We know we need a Poincar\'e inequality for $ u \in H_{\perp}^{1}(M) $, but now the proof is more complicated.
\begin{theorem}\label{Thm 3.5} Fix $ (M,  \partial M, g) $. For every element $ u \in H_{\perp}^{1}(M) $, there exists a positive constant $ C > 0 $, which only depends on $ (M, \partial M, g) $ and a choice of a finite cover on $ M $, such that
\begin{equation}\label{nineteen}
\left\lVert u \right\rVert_{L_{\perp}^{2}(M)} \leqslant C \left\lVert u \right\rVert_{E}. 
\end{equation}
\end{theorem}
\begin{remark} By Definition 3.4, we know that  $ L_{\perp}^{2}(M) $, $ H^1(M) $, and $ H_{\perp}^{1}(M) $ are completions of $ \mathcal{C}^{\perp} $, so it suffices to prove the theorem for $ u \in \mathcal{C}^{\perp} $.
\end{remark}

\noindent The proof of Theorem \ref{Thm 3.5} involves a series of lemmas.  The first technical lemma  drops the compactness assumption on $M$.

\begin{lemma}\label{Lem 3.1} Let $M$ be a smooth, connected, Riemannian manifold with $ \mathcal{C}^{1} $ boundary and a finite cover by coordinate charts. Let $ u \in C^\infty(M)$ satisfy $ \int_{M} u d\mu= 0 $. For any finite coordinate chart cover $ \lbrace U_{i} \rbrace $ of $ M $, we can write $ u = \sum_{i} u_{i} $, where each $ u_{i} $ is smooth, supported in $ U_{i} $, and satisfies 
\begin{equation*}
\int_{U_{i}} u_{i}d\mu = 0.
\end{equation*}
\end{lemma}
\begin{proof} We show this by induction on the number of charts. The case of one chart is trivial.\\
\\
Assume that the statement of lemma holds for $ k - 1 $ charts which covers $ M $. For a finite cover $ \lbrace U_{i} \rbrace_{i=1}^{k} $ by $ k $ charts, we can choose a smooth function $ \sigma $ such that
\begin{equation*}
\text{supp}(\sigma) \subset U_{1} \cap \left(\bigcup_{i=2}^{k} U_{i}\right)\ 
{\rm and}\ \int_{U_{1} \cap (\bigcup_{i=2}^{k} U_{i})} \sigma = 1.
\end{equation*}
\noindent Here we have dropped $d\mu$ from the notation. Take a smooth partition of unity $ \lbrace \chi_{i} \rbrace_{i=1}^{k} $ subordinate to $ \lbrace U_{i} \rbrace_{i=1}^{k} $. Set $ I_{i} = \int_{M} \chi_{i} u $. Since $ \int_{M} u = 0 $, it follows that
\begin{equation*}
\int_{M} \chi_{1} u = -\sum_{i=2}^{k} \int_{M} \chi_{i} u.
\end{equation*}
\noindent Set
\begin{equation*}
u_{1} = \chi_{1} u - I_{1}\sigma, \ u_{1}' = \sum_{i=2}^{k} \chi_{i} u + I_{1} \sigma.
\end{equation*}
\noindent Then $ \text{supp}(u_{1}) \subset U_{1}, \text{supp}(u_{1}') \subset \bigcup_{i=2}^{k} U_{i} $, and we have
\begin{equation*}
\int_{M} u_{1} = \int_{U_{1}} u_{1} = 0,\  \int_{M} u_{1}' = \int_{\bigcup_{i=2}^{k} U_{i}} u_{1}' = 0.
\end{equation*}
\noindent By induction, we can write $ u_{1}' = \sum_{i=2}^{k} u_{i} $, with $ \text{supp}(u_{i}) \subset U_{i}$, and $ \int_{M} u_{i} = 0 $. This completes the induction step.
\end{proof}


\noindent The next lemma is a special case of Schur's Criterion, but has a short proof of independent interest.
\begin{lemma}\label{Lem 3.2} If $ f \in L^{1}(\mathbb{R}^{n}) $ and $ g \in L^{2}(\mathbb{R}^{n}) $, then
\begin{equation}\label{twenty}
\left\lVert f * g \right\rVert_{L^{2}(\mathbb{R}^{n})} \leqslant \left\lVert f \right\rVert_{L^{1}(\mathbb{R}^{n})} \left\lVert g \right\rVert_{L^{2}(\mathbb{R}^{n})}.
\end{equation}
\end{lemma}

\begin{proof} For $x\in \R^n$,  the $ L^{2} $-norm is translation invariant, i.e.,
\begin{equation}\label{transinv}
\left\lVert f(\cdot) \right\rVert_{L^{2}(\mathbb{R}^{n})} = \left\lVert f(\cdot - x) \right\rVert_{L^{2}(\mathbb{R}^{n})},
\end{equation}
\noindent which follows from the translation invariance of Lebesgue measure. Therefore, 
\begin{align*}
\left\lVert f * g \right\rVert_{L^{2}(\mathbb{R}^{n})}^{2} & =\int_{\mathbb{R}^{n}} \left\lvert f * g(x) \right\rvert^{2} dx = \int_{\mathbb{R}^{n}} \left\lvert \int_{\mathbb{R}^{n}} f(y)g(x - y) dy \right\rvert^{2} dx \\
& \leqslant \int_{\mathbb{R}^{n}} \int_{\mathbb{R}^{n}} \int_{\mathbb{R}^{n}} \left\lvert f(y) g(x - y) \right\rvert \left\lvert f(z) g(x - z) \right\rvert dy dz dx \\
& = \int_{\mathbb{R}^{n}} \int_{\mathbb{R}^{n}} \left(\int_{\mathbb{R}^{n}} \left\lvert g(x - y) \right\rvert \left\lvert g(x - z) \right\rvert dx\right) \left\lvert f(y) \right\rvert \left\lvert f(z) \right\rvert dy dz \\
& \leqslant \int_{\mathbb{R}^{n}} \int_{\mathbb{R}^{n}} \left\lvert f(y) \right\rvert \left\lvert f(z) \right\rvert dy dz \cdot \left\lVert g \right\rVert_{L^{2}(\mathbb{R}^{n})}^{2} \\
& \leqslant \left\lVert f \right\rVert_{L^{1}(\mathbb{R}^{n})}^{2} \left\lVert g \right\rVert_{L^{2}(\mathbb{R}^{n})}^{2},
\end{align*}
\noindent where we have used Cauchy-Schwarz and (\ref{transinv}).
\end{proof}
\noindent For the proof of Theorem \ref{Thm 3.5}, we need a local Poincar\'e inequality. The point is that charts on a manifold with boundary fall into two types: interior charts diffeomorphic to open balls in $ \mathbb{R}^{n} $, and boundary charts diffeomorphic to open balls in $ \mathbb{H}^{n} $ which meet $ \partial \mathbb{H}^{n} $. Since the boundary charts determine sets which are not open in $ \mathbb{R}^{n}$, we need a technical version of the Poincar\'e inequality that treats non-open sets.\\
\\
The next two results are modifications of proofs in \cite{Don}, where we drop the usual openness assumption.

\begin{lemma}\label{Lem 3.3} \cite[Thm. 2]{Don} Let $ \Omega $ be a bounded, connected, convex subset of $\mathbb{R}^{n} $and let $ \phi $ be a smooth function on an open, bounded set $ \Omega' \supset \overline{\Omega}.$ Let $ V,d $ be the volume and diameter of $ \Omega $, respectively. Let $ \text{Ave}(\phi) : = \frac{1}{V} \int_{\Omega} \phi dx $  be the average of $ \phi $ with respect to Lebesgue measure. Then for $ x \in \Omega $, we have
\begin{equation}\label{twentyone}
\left\lvert \phi(x) - \text{Ave}(\phi) \right\rvert \leqslant \frac{d^{n}}{nV} \int_{\Omega} \frac{1}{\left\lvert x - y \right\rvert^{n-1}} \left\lvert \nabla \phi(y) \right\rvert dy.
\end{equation}
\end{lemma}
\noindent Since $ \Omega $ is not necessarily open, we need the openness of $ \Omega' $ to define derivatives, especially at the boundary points of $ \Omega $. 
\begin{proof}
By shifting $\Omega$ by $-x$, we may assume that $ x = 0 \in \Omega $. Replacing $\phi$ by $\phi-\phi(0)$, we may also assume
that $ \phi(0) = 0 $. \\
\\
Since $\Omega$ is convex, we may set $ R(\theta) $ to be the length of the portion of the ray at the origin at angular variable $ \theta $ which lies in $ \Omega $. Therefore,
\begin{align*}
\text{Ave}(\phi) &= \frac{1}{V} \int_{0}^{2\pi} \int_{0}^{R(\theta)} \phi(r, \theta) r^{n-1} drd\theta.
 = \frac{1}{V} \int_{0}^{2\pi} \int_{0}^{R(\theta)} \left(\int_{0}^{r} \frac{\partial \phi}{\partial s} ds
 \right) r^{n-1}drd\theta\\
 &= \frac{1}{V} \int_{0}^{2\pi} \int_{s=0}^{R(\theta)} \left(\int_{r = s}^{R(\theta)} r^{n-1} dr\right) \frac{\partial \phi}{\partial s} ds d\theta \\
& = \frac{1}{V} \int_{0}^{2\pi} \int_{s=0}^{R(\theta)} \left(\frac{1}{n} (R(\theta)^{n} - s^{n}\right) \frac{\partial \phi}{\partial s} ds d\theta.
\end{align*}
\noindent By convexity, $ 0 < R(\theta) < d $, and it follows that
\begin{align*}
\left\lvert \text{Ave}(\phi) \right\rvert & \leqslant \frac{1}{V} \int_{0}^{2\pi} \int_{s=0}^{R(\theta)} \frac{d^{n}}{n} \left\lvert \frac{\partial \phi}{\partial s} \right\rvert ds d\theta = \frac{d^{n}}{nV} \int_{0}^{2\pi} \int_{s=0}^{R(\theta)} \left\lvert \frac{\partial \phi}{\partial s} \right\rvert \cdot \frac{1}{s^{n-1}} \cdot s^{n-1} ds d\theta \\
& \leqslant \frac{d^{n}}{nV} \int_{0}^{2\pi} \int_{s=0}^{R(\theta)} \left\lvert \nabla \phi \right\rvert \cdot \frac{1}{s^{n-1}} \cdot s^{n-1} ds d\theta = \frac{d^{n}}{nV} \int_{\Omega} \frac{1}{\left\lvert y \right\rvert^{n-1}} \left\lvert \nabla \phi(y) \right\rvert dy.
\end{align*}
\end{proof}

\begin{proposition} \label{Prop 3.1}
\cite[Cor. 1]{Don}, (Local Poincar\'e Inequality)  In the notation of Lemma \ref{Lem 3.3},  there exists a constant $ C > 0 $, which only depends on $ \Omega \subset \R^n $ and $n$, such that
\begin{equation}\label{twentytwo}
\int_{\Omega} \left\lvert \phi(x) - \text{Ave}(\phi) \right\rvert^{2} dx \leqslant C \int_{\Omega} \left\lvert \nabla \phi \right\rvert^{2} dx.
\end{equation}
\end{proposition}
\begin{proof} Define two functions on $ \mathbb{R}^{n} $, 
\begin{equation*}
F(x) = \begin{cases} \frac{d^{n}}{nV} \frac{1}{\left\lvert x \right\rvert^{n-1}}, & \left\lvert x \right\rvert \leqslant d, \\ 0, & \left\lvert x \right\rvert > d, \end{cases} \ {\rm and}\  G(x) = \begin{cases} \left\lvert \nabla \phi(x) \right\rvert^{2}, & x \in \Omega, \\ 0, & x \notin \Omega. \end{cases}
\end{equation*}
\noindent Clearly $ F \in L^{1}(\mathbb{R}^{n}) $ and $ G \in L^{2}(\mathbb{R}^{n}) $. Using Lemmas \ref{Lem 3.2} and \ref{Lem 3.3},  we  get 
\begin{align*}
\int_{\Omega} \left\lvert \phi(x) - \text{Ave}(\phi) \right\rvert^{2} dx & \leqslant \int_{\Omega} \left(\int_{\Omega} \frac{d^{n}}{nV} \frac{1}{\left\lvert x - y \right\rvert^{n-1}} \left\lvert \nabla \phi(y) \right\rvert dy\right)^{2} dx \\
& \leqslant \int_{\Omega} \left\lvert F * G(x) \right\rvert^{2} dx \leqslant \int_{\mathbb{R}^{n}} \left\lvert F * G(x) \right\rvert^{2} dx \\
& = \left\lvert F * G \right\rvert_{L^{2}(\mathbb{R}^{n})}^{2} \leqslant \left\lvert F \right\rvert_{L^{1}(\mathbb{R}^{n})}^{2} \left\lvert G \right\rvert_{L^{1}(\mathbb{R}^{n})}^{2} \\
& = \left\lvert F \right\rvert_{L^{1}(\mathbb{R}^{n})}^{2} \cdot \left\lvert \nabla \phi \right\rvert_{L^{2}(\mathbb{R}^{n})}^{2}.
\end{align*}
\noindent  $ F $ is integrable and independent of $ \phi $, so setting $ C = \left\lvert F \right\rvert_{L^{1}(\mathbb{R}^{n})}$ proves (\ref{twentytwo}).
\end{proof} 
\noindent Using Lemmas \ref{Lem 3.1} - \ref{Lem 3.3} and Proposition \ref{Prop 3.1}, we are now ready to prove the Poincar\'e inequality in Theorem \ref{Thm 3.5}. \\
\\
{\it Proof of Theorem 3.5.}
To show (\ref{nineteen}), we estimate $ \int_{M} u \phi d\mu $ for $ u, \phi \in \mathcal{C}^{\perp} $.
Recall that in a local chart with coordinates $(x^1,\ldots, x^n)$, the volume form is
$d\mu = \sqrt{\det g} dx,$ where $dx = dx^1\wedge\ldots\wedge dx^n.$  We want $ \sqrt{\det g} $ to be bounded 
on each chart.  In fact, for any open cover $ \lbrace U_{i}^{'} \rbrace $ of $ M $, there exists a new open cover $ \lbrace U_{i} \rbrace $ of $ M $, each $ \overline{U_{i}} \subset U_{i}^{'} $ \cite[Ch. 13]{Lee}. As a result,
there exists $D_i$ such that $\sqrt{\det g}< D_i$ is bounded on each $U_i.$\\
\\
By Lemma \ref{Lem 3.1}, we can write
\begin{equation*}
\int_{M} u \phi d\mu =  
\sum_{i=1}^{k} \int_{U_{i}} u_{i} \phi d\mu \ {\rm and} \int_{U_i} u_i d\mu = 0.
\end{equation*}
Each $ U_{i} $ is diffeomorphic to  an open ball, i.e. a convex, bounded and connected subset $V_i$ in either $ \mathbb{R}^{n} $ or $ \mathbb{H}^{n} $. We  denote the functions on $V_i$ corresponding to $u_i, \phi,$ by $ u'_{i}, \phi_i', $ 
respectively. \\
\\
On each $ U_{i} $, we have
\begin{align*}
\lVert \phi - \text{Ave}(\phi)) \rVert_{L^2(U_i)} & = \left\lVert \phi' - \text{Ave}(\phi')\right\rVert_{L^2(V_i, g)} = \int_{V_i}\left\lvert \phi' - \text{Ave}(\phi')\right\rvert^{2} \sqrt{\det(g)} dx \\
& \leqslant D_{i} \int_{V_{i}} \left\lvert \phi' - \text{Ave}(\phi') \right\rvert^{2} dx.
\end{align*}
\noindent By Lemma \ref{Lem 3.1}, we can write $u = \sum_i u_i$ with  $0= \int_{U_{i}} u_{i} =\text{Ave}(\phi) \cdot\int_{U_{i}} u_{i} $, so by Proposition \ref{Prop 3.1}
\begin{align*}
\left\lvert \int_{M} u\phi d\mu \right\rvert & \leqslant \sum_{i=1}^{k} \left\lvert \int_{U_{i}} u_{i} \phi d\mu \right\rvert = \sum_{i=1}^{k} \left\lvert \int_{U_{i}} u_{i} (\phi - \text{Ave}(\phi)) d\mu \right\rvert \\
& \leqslant \sum_{i=1}^{k} \left\lVert u_{i} \right\rVert_{L_{\perp}^{2}(U_{i})} \left\lVert \phi - \text{Ave}(\phi) \right\rVert_{L^{2}(U_{i})} \leqslant \sum_{i=1}^{k} \left\lVert u_{i} \right\rVert_{L_{\perp}^{2}(U_{i})}\cdot D_i \left\lVert \nabla \phi \right\rVert_{L^{2}(U_{i})}\\
& \leqslant C \left\lVert u \right\rVert_{L_{\perp}^{2}(M)} \left\lVert \phi \right\rVert_{E^{'}}.
\end{align*}
\noindent Here $ C > 0 $ only depends on $ (M, g) $ and the finite cover we choose. 
The last inequality is not immediate. We justify it in the situation where $ M $ is covered by two charts, and the general case follows inductively as in the proof of Lemma \ref{Lem 3.1}. \\
\\
\noindent 
First, it is clear that $ \lVert \nabla \phi \rVert_{L^{2}(U_{i})} \leqslant \lVert \phi \rVert_{E^{'}} $. Now for $ \lVert u_{i} \rVert_{L_{\perp}^{2}}, i = 1, 2 $, as in Lemma \ref{Lem 3.1} we can write
\begin{equation*}
    u_{1} = \chi_{1}u - I_{1}\sigma, u_{2} = \chi_{2} u + I_{1} \sigma.
\end{equation*}
\noindent Thus we can estimate that
\begin{align*}
     \lVert u_{1} \rVert_{L_{\perp}^{2}(U_{1})} &\leqslant \lVert \chi_{1} u \rVert_{L^{2}(U_{1})} + \lVert I_{1}\sigma \rVert_{L^{2}(U_{1})} \leqslant \lVert \chi_{1} u \rVert_{L^{2}(M)} + \lvert I_{1} \rvert \lVert \sigma \rVert_{L^{2}(U_{1})} \\
    &\leqslant  \lVert \chi_{1} \rVert_{L^{2}(M)} \lVert u \rVert_{L_{\perp}^{2}(M)} + \left| \int_{U_{1}} \chi_{1} u \rvert \lVert \sigma \rVert_{L^{2}(M)}\right|\\
    &\leqslant \lVert \chi_{1} \rVert_{L^{2}(M)} \lVert u \rVert_{L_{\perp}^{2}(M)} + \int_{U_{1}} \lvert \chi_{1} \rvert \lvert u \rvert \cdot \lVert \sigma \rVert_{L^{2}(M)} \\
    &\leqslant  \lVert \chi_{1} \rVert_{L^{2}(M)} \lVert u \rVert_{L_{\perp}^{2}(M)} ( 1 + \lVert \sigma \rVert_{L^{2}(M)}) \leqslant C_{1} \lVert u \rVert_{L_{\perp}^{2}(M)}.
\end{align*}
\noindent Here $ C_{1} $ only depends on the partition of unity as well as the choice of $ \sigma $. Similarly, we  obtain
\begin{equation*}
    \lVert u_{2} \rVert_{L_{\perp}^{2}(U_{2})}  \leqslant C_{2} \lVert u \rVert_{L^{2}(M)}.
\end{equation*}
\noindent Thus there exists  $ C>0 $ such that
\begin{equation*}
    \left| \int_{M} u\phi d\mu \right| \leqslant C \lVert u \rVert_{L_{\perp}^{2}(M)} \lVert \phi \rVert_{E^{'}}.
\end{equation*}
\noindent Lastly, we choose $ \phi = \bar{u} $, then we have
\begin{equation*}
\left\lVert u \right\rVert_{L_{\perp}^{2}(M)}^{2} \leqslant C \left\lVert u \right\rVert_{L_{\perp}^{2}(M)} \left\lVert u \right\rVert_{E}.
\end{equation*}
\noindent This finishes the proof of Theorem \ref{Thm 3.5}. $ \Box $

\begin{remark} Note that we do not need to specify a boundary condition  in the proof of Theorem \ref{Thm 3.5}. Note also that the Poincar\'e constant in Theorem \ref{Thm 3.5} may not be the first positive eigenvalue of $ - \Delta_{g} $ on $ M $.
\end{remark}
\noindent The integration by parts in (\ref{motivation}) is still valid for functions with Neumann boundary condition.  This motivates the definition of a weak solution in this setting.

\begin{definition} $ u \in H_{\perp}^{1}(M) $ is a weak solution of the Neumann boundary problem $ -\Delta_{g}u = f, g(\nabla u, N) = 0 \; on \; \partial M $ if
\begin{equation*}
B_{\mathcal{N}}[v, u] : = \langle v, u \rangle_{E} = \langle v, f \rangle_{L^{2}},
\end{equation*}
\noindent for all $ v \in H_{\perp}^{1}(M) $.
\end{definition}

\noindent We will need the complex-valued version of the Lax-Milgram Theorem.
\begin{theorem} {\it (Lax-Milgram)} Let $ \mathcal{H} $ be a complex Hilbert space with inner product $ \langle \cdot, \cdot \rangle $ and norm $ \left\lVert \cdot \right\rVert $. Let $ B : \mathcal{H} \times \mathcal{H} \rightarrow \mathbb{C} $ be a sesquilinear map satisfying 
\begin{align*}
 \left\lvert B[u, v] \right\rvert &\leqslant \alpha \left\lVert u \right\rVert \left\lVert v \right\rVert
 \ \ \  (boundedness),\\
 \left\lvert B[u, u] \right\rvert &\geqslant \beta \left\lVert u \right\rVert^{2} \ \ \ (Coercivity),
\end{align*}
\noindent for  constants $ \alpha, \beta > 0 $ and for all $u,v\in\mathcal H.$ Then for any bounded linear functional $ L : \mathcal{H} \rightarrow \mathbb{C} $, there exists a unique element $ u \in \mathcal{H} $ such that
\begin{equation}\label{eighteen}
B[v, u] = L(v), \forall v \in H.
\end{equation}
\end{theorem}
\begin{proof}
See \cite[Sec.~6, Thm.~6]{Lax}.
\end{proof}
\noindent We now combine this theorem with the Poincar\'e inequality to prove the existence and uniqueness of weak solutions to the linear Poisson equation.
\begin{theorem}\label{Thm 3.7} For $ (M, \partial M, g) $ and $ f \in L^{2}(M) $, there exists a unique weak solution $ u \in H_{\perp}^{1}(M) $ of the Neumann boundary problem
\begin{equation}\label{twentythree}
-\Delta_{g}u = f, g(\nabla u, N) = 0 \; on \; \partial M.
\end{equation}
\end{theorem}

\begin{proof} It is immediate that $L:H_\perp^1(M)\to\C, L(v) = \langle v, f\rangle_{L^2(M)}$ is a bounded linear functional. For the sesquilinear map $ B_{\mathcal{N}} : H_{\perp}^{1}(M) \times H_{\perp}^{1}(M) \rightarrow \mathbb{C} $, we have
\begin{equation*}
\left\lvert B_{\mathcal{N}}[v, u] \right\rvert = \left\lvert \langle v, u \rangle_{E} \right\rvert \leqslant 
 \left\lVert u \right\rVert_E \left\lVert v \right\rVert_E  
\leqslant
\left\lVert u \right\rVert_{H_\perp^1(M)} \left\lVert v \right\rVert_{H_{\perp}^{1}(M)}.
\end{equation*}
\noindent Hence the boundedness condition for $ B_{\mathcal{N}}[v, u] $ is satisfied. 
For the constant $C$ in the Poincar\'e inequality (\ref{nineteen}), the coercivity condition follows from
\begin{align*}
\left\lvert B_{\mathcal{N}}[u, u] \right\rvert & = \left\lVert u \right\rVert_E^{2} = \frac{C^{2}}{1 + C^{2}} \left\lVert u \right\rVert_{E}^{2} + \frac{1}{1 + C^{2}} \left\lVert u \right\rVert_{E}^{2} \geqslant  \frac{C^{2}}{1 + C^{2}} \left\lVert u \right\rVert_{E}^{2} + \frac{1}{1 + C^{2}} C^{2} \left\lVert u \right\rVert_{L_{\perp}^{2}(M)}^{2} \\
& = \frac{C^{2}}{1 + C^{2}}(\left\lVert u \right\rVert_{E}^{2} + \left\lVert u \right\rVert_{L_{\perp}^{2}(M)}^{2}) \geqslant \frac{C^{2}}{2(1 + C^{2})} \left\lVert u \right\rVert_{H_{\perp}^{1}(M)}^{2}.
\end{align*}
\noindent By Lax-Milgram, we conclude that there exists a unique $ u \in H^{1}(M) $ such that $ B_{\mathcal{N}}[v, u] = \langle v, f \rangle_{L^{2}}, \forall v \in H_{\perp}^{1}(M) $.
\end{proof}
\noindent We can now solve the nonlinear Poisson equation (\ref{seventeen}) by combining Theorem \ref{Thm 3.7} with the iterative method.
\begin{theorem} Fix $ (M, \partial M, g) $ and $f\in L^2(M).$ Let $ C $ be the Poincar\'e constant on $ M $  in Theorem \ref{Thm 3.5}. Assume that $ F : \mathbb{C} \rightarrow \mathbb{C} $ is a globally Lipschitz function with Lipschitz constant $ C_{1} $ such that $ C_{1}C^{2} < 1 $. Then (\ref{seventeen}) has a weak solution $ u \in H_{\perp}^{1}(M) $.
\end{theorem}

\begin{proof} This proof is essentially the same as in Theorems \ref{Thm 2.1} and \ref{Thm 3.4}.

\end{proof}


\subsection{Nonlinear Laplace Equations on Closed Manifolds}
It is technically difficult to apply the iterative method to closed Riemannian manifolds $(M,g)$. We show how to surmount these difficulties on the $n$-sphere $S^{n} $, and discuss the obstacles to this approach on general closed manifolds. \\
\\
Consider the strong version of the iterative steps (\ref{twelve}): 
\begin{align*}
 -\Delta_{g} u_{0} &= f, \\
 -\Delta_{g} u_{k} + F(u_{k-1}) &= f,
\end{align*}
\noindent for $k\in \N.$ Since $\int_M -\Delta_g f = \int_M -\Delta_g f \cdot 1 = \int_M f\cdot -\Delta_g 1 = 0$ on a closed manifold,
we must have
\begin{equation*}
\int_{M} f = 0, \int_{M} (f - F(u_{k-1}) )= 0, 
\end{equation*}
\noindent which implies that $ \int_{M} F(u_{k-1}) \equiv 0 $. This is difficult to guarantee.  It follows that simply mimicking the iterative method fails on closed manifolds. However, we can reintroduce the Dirichlet problem by considering the PDE chart by chart.\\
\\
\noindent 
This works on $S^n$, because the chart structure is particularly simple. We can cover $S^n$ by smooth charts $ U_{1} $ and $ U_{2} $, where $ U_{1} $  contains the north pole and is a bit larger than the upper hemisphere and $ U_{2} $  contains the south pole and is a bit larger than the lower hemisphere.  Hence $ U_{1} \cap U_{2} $ contains the equator and is diffeomorphic to $ S^{n-1} \times \mathbb{R} $. We also have $ \partial U_{1} \cap \partial U_{2} = \emptyset $. With this setup, we can apply artificial boundary conditions as in \cite{UV} for the charts $ U_{1} $ and $ U_{2} $ to obtain local solutions and then glue them together.  
\begin{theorem} Fix a Riemannian metric $g$ on $S^n$, and fix $f\in L^2(S^n).$ Set $ C  = \max\{D_1, D_2\}$, where $D_i$ is  the Poincar\'e constant for $U_i$. Assume that $ F : \mathbb{C} \rightarrow \mathbb{C} $ is a globally Lipschitz function with Lipschitz constant $ C_{1} $ such that $ C_{1}C^{2} < 1 $. Then
\begin{equation*}
-\Delta_{g} u + F(u) = f
\end{equation*}
\noindent has a weak solution $ u \in H^{1}(S^{n}) $. 
\end{theorem}
\begin{proof} On $ U_{1} $, we choose $h\in C^\infty(\partial U_1)$ and consider the PDE
\begin{equation*}
-\Delta_{g} u_{1} + F(u_{1}) = f \; \text{on} \; U_{1}, u_{1} |_{\partial U_{1}} = h.
\end{equation*}
\noindent $ U_{1} $ is an open subset of $S^{n} $, so it is itself a Riemmanian manifold with Riemannian metric given by restricting $ g $ to $ U_{1} $. Since $C_1D_1^2 <1,$ by Corollary 3.1 we have a solution $ u_{1} $ on $ U_{1} $. Similarly, we consider
\begin{equation*}
-\Delta_{g} u_{2} + F(u_{2}) = f \; \text{on} \; U_{2}, u_{2} |_{\partial U_{2}} = u_{1} |_{\partial U_{2}}.
\end{equation*}
\noindent By Corollary 3.1 again, we conclude that there exists $ u_{2} $ that solves this PDE  on $ U_{2} $. We now set
\begin{equation*}
v = \begin{cases} & u_{1} \; \text{on} \; U_{1} \backslash \overline{U_{1} \cap U_{2}}, \\ & u_{2} \; \text{on} \; U_{2}, \\ & u_{1} |_{\partial U_{2}} \; \text{on} \; \partial U_{2}. \end{cases}
\end{equation*}
\noindent By the trace theorem, we conclude that $ v \in L^{2}(S^{n}) $ solves the PDE $ -\Delta_{g} u + F(u) = f $ on $ \mathbb{S}^{n} $ weakly.
\end{proof}
\noindent For a general closed manifold, it seems technically difficult to find a cover by charts 
$\{U_i\}_{i=1}^k$ such that sets like $U_1\setminus \overline{\left(\cup_{i=2}^k U_i\right)}$ have $\calC^1$ boundary.  This prevents us from using the Poincar\'e inequality Theorem \ref{Thm 3.2} and Corollary 3.1.


\section{Parametrices for Nonlinear Laplace Equations}
The powerful parametrix method to solve linear elliptic differential operators, first discussed by Hadamard (see e.g. \cite[Vol. 3, Ch. 17]{Hor2}) has been applied to many other problems in analysis, topology and geometry through the development of microlocal analysis in e.g. \cite{Hor2}. In this section, we discuss extending the parametrix method to nonlinear Laplace equations on both Euclidean domains $\Omega$ and compact manifolds $M$ with boundary. We obtain an integrated version of $-\Delta_gu + F(u) = f, u|_{\partial M} = 0$, which leads to refined estimates for the solutions to the Dirichlet problem given in Theorems \ref{Thm 2.1} and \ref{Thm 3.4}.  \\
\\ 
In this section, we assume familiarity with Sobolev spaces $H^s(\Omega)$ and $H^s(M)$, for $s\in\R$. In particular, for $s>0$, the Sobolev pairing of $ H^{-s}(\Omega), H^{-s}(M) $ with $H^s(\Omega), H^s(M)$, resp., identifies $ H^{-s}(\Omega), H^{-s}(M) $ with the dual space of 
$ H^{s}(\Omega), H^{s}(M) $, resp.
 We also assume familiarity with symbols and principal symbols of differential operators. \\
 \\
 \noindent For a domain $\Omega\in \R^n$, let $ \mathcal{C}^{-\infty}(\Omega) $ be the space of distributions on $\Omega$. 
 
\begin{definition} An operator $ R : \mathcal{C}^{-\infty}(\Omega) \rightarrow \mathcal{C}^{-\infty}(\Omega) $ is properly supported if for all $ u \in \mathcal{C}^{-\infty}(\Omega) $ with support $ \text{supp}(u)$ contained in a compact subset $K\subset \Omega$, there exists a compact subset $K' = K'(K, u)\subset \Omega$ such that $ \text{supp}(Ru)\subset K'$.
\end{definition}
\noindent Any properly supported smooth operator acting on $ H^{s}(\Omega) $ is a compact operator. \\
\\
By work of H\"ormander \cite[Vol.~3, Ch.~20]{Hor2} and Melrose \cite[Ch.~4]{RBM1}, we have the following theorem for elliptic differential operators with constant coefficients. 
\begin{theorem} Let $ \Omega $ be an open subset of $ \mathbb{R}^{n} $, and let $ D $ be an elliptic differential operator with constant coefficients and order $m$. There exists a linear operator $ Q_{\Omega} : \mathcal{C}^{-\infty}(\Omega) \rightarrow \mathcal{C}^{-\infty}(\Omega), $ which is bounded  from $ H_{loc}^{s}(\Omega) $ to $ H_{loc}^{s + m}(\Omega) $ for each $ s \in \mathbb{R} $, which gives a 2-sided parametrix for $ D $ in $ \Omega $:
\begin{equation}
D Q_{\Omega} = I - R, Q_{\Omega} D = I - R',
\end{equation}
where $ R $ and $ R' $ are properly supported operators which are bounded  between any two Sobolev spaces
$H^s_{loc}(\Omega), H^{s'}_{loc}(\Omega)$. 
\end{theorem}
\noindent $Q_\Omega$ is called the parametrix of $D.$
\begin{remark} 
If $\Omega$ is precompact, we can replace $ H_{loc}^{s}(\Omega), H_{loc}^{s + m}(\Omega) $ in the Theorem by $ H^{s}(\Omega),  H^{s + m}(\Omega) $ by the standard technique of extending $u\in H^s(\Omega)$, $u|_{\partial \Omega} = 0$, by zero to an open bounded set containing $\Omega.$
\end{remark}
\noindent  We now use the parametrix $Q_\Omega$ to reformulate the Dirichlet problem (\ref{Thm 2.1 PDE}) in order to obtain new estimates for solutions. Since $Q_\Omega$ and $R'$ are pseudodifferential operators, (\ref{twentyfour}) replaces the differentiations in (\ref{Thm 2.1 PDE}) with integrations. While this does not make solving (\ref{Thm 2.1 PDE}) easier, it does lead to the improved estimates (\ref{new estimate}). 

\begin{proposition} Consider the Dirichlet problem on a bounded domain $\Omega$ under the hypotheses in Theorem \ref{Thm 2.1}. There exists a parametrix $ Q_{\Omega} : L^{2}(\Omega) \rightarrow H_{0}^{2}(\Omega) $ of $ -\Delta $ such that
\begin{equation}\label{twentyfour}
u = Q_{\Omega} f + R'u - Q_{\Omega} F(u).
\end{equation}
\noindent If the Lipschitz constant $ C $ of $F$ satisfies $ C \left\lVert Q_{\Omega} \right\rVert_{L^{2} \rightarrow H^{2}} < 1 $, then  for $s\in \R$, we have
\begin{equation}\label{new estimate}
\lVert u \rVert_{H_{0}^{1}(\Omega)} \leqslant C' \left \lVert f \right\rVert_{L^{2}(\Omega)} + C_{s} \left\lVert u \right\rVert_{H^{s}(\Omega)},
\end{equation}
\noindent for some positive constants $ C' $ and $ C_{s} $.  
\end{proposition}
\begin{proof} The Euclidean Laplacian $ - \Delta $ is clearly an elliptic operator a constant coefficient elliptic operator, so by Theorem 4.1, as well as Remark 4.1,  there exists a 2-sided parametrix $ Q_{\Omega} $ of $ - \Delta $ such that $ Q_{\Omega} (-\Delta) = Id - R' $, where $ R' $ is properly supported. Applying $ Q_{\Omega} $ to both sides of (3) and (4), we observe that
\begin{align}
 Q_{\Omega} (-\Delta u_{0}) &= Q_{\Omega} f \Rightarrow u_{0} = Q_{\Omega} f + R' u_{0}; \nonumber \\
 Q_{\Omega} (-\Delta u_{k}) + Q_{\Omega} (F(u_{k-1}) &= Q_{\Omega} f \Rightarrow u_{k} = Q_{\Omega} f + R' u_{k - 1} -  Q_{\Omega} (F(u_{k-1})). \label{Parametrix Iteration}
\end{align}
\noindent By (5) and the hypotheses of Theorem \ref{Thm 2.1}, we know that $ f, F(u_{k-1}) \in L^{2}(\Omega) $. By Theorem 4.1, we conclude that $ Q_{\Omega} f, Q_{\Omega} (F(u_{k-1})) \in H^{2}(\Omega) $ and $ R' u_{k - 1} \in C^{\infty}(\Omega) $. We can therefore take the limit for the sequence $ \lbrace u_{k} \rbrace $ in $ H_{0}^{1}(\Omega) $. By continuity of $Q_\Omega, R', F $, we take limit on both sides of (\ref{Parametrix Iteration}) to get 
\begin{align*}
u &= Q_{\Omega} f + R'u - Q_{\Omega} F(u).
\end{align*}
\noindent Therefore,
\begin{equation*}
\lVert u \rVert_{H_{0}^{1}(\Omega)} \leqslant \lVert Q_{\Omega} f \rVert_{H_{0}^{1}(\Omega)} + \lVert R' u \rVert_{H_{0}^{1}(\Omega)} + \lVert Q_\Omega \rVert_{L^{2}(\Omega) \rightarrow H^{2}(\Omega)} C \lVert u \rVert_{H_{0}^{1}(\Omega)},
\end{equation*}
\noindent where the estimate in the last term follows as in (5).Since $ C \left\lVert Q_{\Omega} \right\rVert_{L^{2}(\Omega) \rightarrow H^{2}(\Omega)} < 1 $, we can shift the last term to the left hand side  and conclude that
\begin{equation*}
\lVert u \rVert_{H_{0}^{1}(\Omega)} \leqslant C' \left \lVert f \right\rVert_{L^{2}(\Omega)} + C_{N} \left\lVert u \right\rVert_{H^{-N}(\Omega)}, \forall N \in \mathbb{Z}
\end{equation*}
for positive constants $ C' $ and $ C_{N} $. $ C_{N} $ depends on the choice of the Sobolev space $ H^{-N}(\Omega) $. 
\end{proof}
\noindent We want to extend the estimate (\ref{new estimate}) to the Dirichlet problem on manifolds with boundary.  To do so, we need to recall results of H\"ormander \cite[Vol. 3, Ch. 20]{Hor2} and his definition of elliptic boundary value problems for operators on sections of vector bundles.\\
\begin{definition}  (i) Let $ M $ be a smooth manifold. Assume we have an altas $ \lbrace (U_{\alpha}, \phi_{\alpha}) \rbrace $ of $ M $ and distributions $ u_{\alpha} \in \mathcal{C}^{-\infty}(\phi_{\alpha}(U_{\alpha})) $ such that
\begin{equation*}
u_{\beta} = (\phi_{\alpha} \circ \phi_{\beta}^{-1})^{*} u_{\alpha} \; \text{on} \; \phi_{\beta}(U_{\alpha} \cap U_{\beta}).
\end{equation*}
\noindent Then $ u = \lbrace u_{\alpha} \rbrace $ is called a distribution on $ M $, denoted by $ u \in \mathcal{C}^{-\infty}(M) $. \\
\\
(ii) Let $ \mathcal{C}^{-\infty}(\Omega, \Omega \times \mathbb{C}^{n})$ be the space of distributions on $C_c^\infty(\Omega,  \C^n)$; we can identify $C_c^\infty(\Omega,  \C^n)$ with the space of compactly supported smooth sections of the trivial bundle $\Omega\times \C^n\to \Omega.$\\
\\
(iii) Let $ E \xrightarrow{\pi} M $ be a complex rank $n$ vector bundle 
with
local trivializations $ \lbrace (\pi^{-1}(U_{\alpha}), \phi_{\alpha}) \rbrace $,
$\phi_\alpha:\pi^{-1}(U_\alpha) \to U_\alpha \times \C^n$. 
Given distributions $ u_{\alpha} \in  C^{-\infty}(U_{\alpha},U_\alpha\times \C^n)$ such that
\begin{equation*}
    u_{\beta} = (\phi_{\alpha} \circ \phi_{\beta}^{-1})^{*} u_{\alpha} \; \text{on} \; \phi_{\beta}(\pi^{-1}(U_{\alpha} \cap U_{\beta})),
\end{equation*}
we call $u = \lbrace u_{\alpha} \rbrace $ a distribution on $ E \xrightarrow{\pi} M $, denoted by $ u \in \mathcal{C}^{-\infty}(M, E) $.
\end{definition}
\noindent Let $(M, \partial M, g)$ be a compact, smooth Riemannian manifold with $\calC^1$ boundary. Let $ D : \mathcal{C}^{\infty}(M, E) \to \mathcal{C}^{\infty}(M, F)$ be an order $m$ elliptic differential operator acting between smooth sections of complex vector bundles $E, F$ of the same rank $N$ over $M$. Let $ B_{j} : \mathcal{C}^{\infty}(M, E) \rightarrow \mathcal{C}^{\infty}(\partial M, G_{j}), j = 1, \dotso, J $ be differential operators of order $ m_{j} $, where $ G_{j} $ are smooth, complex vector bundles over $ \partial M $. The boundary problem for $ u \in \mathcal{C}^{\infty}(M, E) $  is
\begin{equation}\label{twentyfive}
Du = f \; in \; M; B_{j}u = g_{j} \; in \; \partial M, j = 1, \dotso, J,
\end{equation}
\noindent where $ f \in \mathcal{C}^{\infty}(M, F) $ and $ g_{j} \in \mathcal{C}^{\infty}(\partial M, G_{j}) $ are given. In fact, we will only need the trivial case 
$E=F=M\times \C$, in which case $C^\infty(M,E) = C^\infty(M,F) = C^\infty(M,\C).$ \\
\\
Let $D_t = \frac{1}{i}\partial_t$, where $t$ is the globally defined inward pointing normal coordinate on $\partial M.$ For a vector bundle $E$ over $M$, let $E_x $ be the fiber at $x\in M$.
\begin{definition} \cite[Vol. 3, Def. 20.1.1]{Hor2} The boundary problem (\ref{twentyfive}) is called elliptic if \\
(i) $ D $ is elliptic; \\
(ii) The boundary conditions are elliptic in the sense that for every $ x \in \partial M $ and $ \zeta \in T_{x}^{*}(M) $ not proportional to the interior conormal vector $ n_{x} $ at $ x\in \partial M, $ the map
\begin{equation}\label{twentysix}
M_{x, \xi}^{+} \rightarrow (G_1)_x \oplus \dotso \oplus (G_J)_x, u \mapsto (b_{1}(x, \xi + D_{t}n_{x}) u(0), \dotso, b_{J}(x, \xi + D_{t}n_{x})u(0)).
\end{equation}
\noindent is bijective, where
\begin{equation*}
M_{x,\xi}^{+} = \lbrace u \in \mathcal{C}^{\infty}(\mathbb{R}_{+}, E_{x}): p(x, \xi + D_{t}n_{x})u(t) = 0, 
\sup_{t \in \mathbb{R}^{+}}\left\lvert u(t) \right\rvert < \infty, \forall t \in \mathbb{R}_{+} \rbrace.
\end{equation*}
\noindent Here $ b_{i}, i = 1, \dotso, n $ are principal symbols of differential operators $ G_{i}, i = 1, \dotso, n $, respectively.
\end{definition}
\noindent In fact, all elements in $ M_{x, \xi}^{+} $ are  exponentially decreasing on $ \mathbb{R}_{+} $ \cite[Vol.~3, Ch.~20]{Hor2}. 
\begin{definition} The space $ \bar{H}^{s}(M^{\circ}, E) $ is defined to be
\begin{equation*}
\bar{H}^{s}(\mathring M, E) = \lbrace u \in \mathcal{C}^{-\infty}(M, E) : u |_{\mathring M} \in H^{s}(M, E) \rbrace.
\end{equation*}
\noindent Here $ \mathring M $ is the interior of the smooth manifold $ X $ with boundary.
\end{definition}


\noindent
The theorem of H\"ormander corresponds to the Fredholm property of $ D : \bar{H}^{s}(\mathring M, E) \rightarrow \bar{H}^{s - m}(\mathring X, F) $ in (\ref{twentyfive}) is given as follows.

\begin{theorem}\label{Thm 4.2} \cite[Vol. 3, Thm. 20.1.2]{Hor2} If the boundary problem (\ref{twentyfive}) is elliptic and $ s \geqslant m $, then
\begin{align*}
\bar{H}^{s}(\mathring M, E) \ni & u \mapsto (Du, B_{1}u, \dotso B_{J}u) \in \\
& \bar{H}^{s - m}(\mathring M, F) \oplus H^{s - m_{1} - \frac{1}{2}}(\partial M, G_{1}) \oplus \dotso \oplus H^{s - m_{J} - \frac{1}{2}}(\partial M, G_{J})
\end{align*}
\noindent is a Fredholm operator.
\end{theorem}
\noindent We can apply H\"ormander's method above to show that we can obtain an integration version of Dirichlet problem (\ref{Poisson manifold}) $ -\Delta_{g} u + F(u) = f, u \equiv 0 \; on \; \partial M $ on compact manifold with boundary by choosing an appropriate Riemannian metric and assuming some regularity assumption for inhomogeneous term. \\
\\
To start, let $ \lbrace U_{i} \rbrace $ be boundary charts of $ M $. We can assume that the defining function on $ U_{i} \cap \partial M $ are $ x_{n} = 0 $ in local coordinates, so locally within each boundary chart $ U_{i} $, we choose a Riemannian metric $ g $ which is of the form $ dx_{n}^{2} + h_{i} $ on each boundary chart $ U_{i} $ where $ h_{i} $ is a metric on $ U_{i} \cap \partial M $. We then consider (\ref{Poisson manifold}) on $ (M, g) $ when $ f \in \mathcal{C}^{\infty}(M) $ so the method in Section 3.2 concludes that there exists a smooth solution of $ (\ref{Poisson manifold}) $. It follows from the procedure and we conclude that we can construct a smooth sequence $ u_{k} $ from iterative steps (\ref{twelve}) which converges to $ u $ in $ H_{0}^{1}(M) $ - norm. 
\begin{proposition} Consider the Dirichlet problem
\begin{equation}\label{twentynine}
-\Delta_{g} u + F(u) = f, u \equiv 0 \; on \; \partial M,
\end{equation}
\noindent on a smooth, connected, compact Riemannian manifolds $ (M, g) $ with $ \mathcal{C}^{1} $ boundary $ \partial M $. Here $ F $ is a complex-valued globally Lipschitz function with Lipschitz constant $ C_{1} > 0 $, $ g $ is the Riemannian metric discussed above and $ f \in \mathcal{C}^{\infty}(M) $. There exists a generalized inverse $ Q : C^{\infty}(M) \rightarrow C^{\infty}(M) $ of $ -\Delta_{g} $ on $ M $ such that the equation
\begin{equation}\label{twentyeight}
u = Ru - QF(u) + Qf.
\end{equation}
\noindent holds. If in particular the Lipschitz constant $ C $ satisfies $ \lVert Q \rVert_{L^{2}(M) \rightarrow H^{2}(M)} C < 1 $, we have
\begin{equation*}
\lVert u \rVert_{H_{0}^{1}(M)} \leqslant C' \lVert f \rVert_{L^{2}(M)} + C_{N} \lVert u \rVert_{H^{-N}(M)}, \forall N \in \mathbb{Z}
\end{equation*}
\noindent with some positive constants $ C' $ and $ C_{N} $.
\end{proposition}
\begin{proof} Recall the proof in Theorem 3.4. With the higher regularity here, the sequence $ u_{k} \in \mathcal{C}^{\infty}(M) $ from iteration steps (\ref{twelve}) can be considered as smooth sections of the trivial bundle $ M \times \mathbb{C} \xrightarrow{\pi} M $ where $ M $ is a compact smooth manifold with boundary $ \partial M $. In iteration, we equivalently consider the PDE
\begin{equation}\label{twentyseven}
-\Delta_{g} u_{k} + F(u_{k-1}) = f, Bu_{k} \equiv 0, k \in \mathbb{Z}_{\geqslant 0}.
\end{equation}
\noindent Here $ B: \mathcal{C}^{\infty}(M, M \times \mathbb{C}) \rightarrow \mathcal{C}^{\infty}(\partial M, \partial M \times \mathbb{C}) $ is the trivial boundary operator, which corresponds to the Dirichlet boundary condition. We now check with Definition 4.4 that the Laplacian with Dirichlet boundary condition is elliptic as a boundary value problem. \\
\\
Since $ B $ is the trivial operator, the zero operator, it follow that the RHS of (\ref{twentysix}) is constantly zero, so we must show that under this situation, the set $ M_{x, \xi}^{+} $ contains only one trivial element, the zero element. It suffices to show that
\begin{equation*}
p(x, \xi + D_{t}n_{x}) u(t) = 0 \Rightarrow u(t) \equiv 0,
\end{equation*}
\noindent whose solutions are also exponentially decay in $ \mathbb{R}_{+} $, as Remark 4.4 suggested. The conormal direction here is $ n_{x} = (0, \dotso, 0, 1) $. The symbol of $ -\Delta_{g} $ is of the form $ Q(\xi_{1}, \dotso, \xi_{n-1}) + \xi_{n}^{2} $ where $ Q $ is a quadratic polynomial of the first $ n - 1 $ variables in cotangent bundle, since the Riemannian metric is of the form $ g = dx_{n}^{2} + h_{i} $ locally in each boundary chart. It follows that we are solving the following ODE
\begin{equation*}
Q(\xi_{1}, \dotso, \xi_{n-1})u(t) + (\xi_{n} -i\frac{d}{dt})^{2}u(t) = 0 \Rightarrow \frac{d^{2}u}{dt^{2}} + 2i\xi_{n} \frac{du}{dt} - (\xi_{n}^{2} + Q(\xi_{1}, \dotso, \xi_{n-1}))u = 0.
\end{equation*}
\noindent The solution of this ODE is given by
\begin{equation*}
u(t) = ae^{\frac{\sqrt{Q(\xi_{1}, \dotso, \xi_{n-1})}t - i\xi_{n}t}{2}} + be^{\frac{\sqrt{-Q(\xi_{1}, \dotso, \xi_{n-1})}t - i\xi_{n}t}{2}}, u(0) = 0 \in M_{x, \xi}^{+}.
\end{equation*}
\noindent where $ a, b $ are some constants related to initial condition. Since $ u(0) = 0 $, then it follows that either $ a = b \equiv 0 $ or $ a = -b \neq 0 $. In the former case, $ u(t) \equiv 0 $. The latter case is impossible, since the solution should be exponentially decaying. Hence the only exponentially decaying solution of this ODE is $ u(t) \equiv 0 $, and it follows that Definition 4.4 is satisfied for (\ref{twentyseven}), which is a linear equation for each fixed $k$. \\
\\
Since Definition 4.4 is satisfied, it follows from Theorem 4.2 that the map
\begin{equation*}
(-\Delta_{g}, B) : \bar{H}^{s}(\mathring M, M \times \mathbb{C}) \rightarrow \bar{H}^{s - 2}(\mathring M, M \times \mathbb{C}) \oplus H^{s - 2}(\partial M, \partial M \times \mathbb{C}),
\end{equation*}
\noindent in  (\ref{twentyseven}) are Fredholm operators, which implies that $ -\Delta_{g} $ has a parametrix $ Q $ such that $ -\Delta_{g} Q = Id - R $, $ Q (-\Delta_{g}) = Id - R' $ for some compact operators $ R, R' $ which are bounded between any two Sobolev spaces. \\
 \\
Due to the trivial Dirichlet boundary condition, \cite{RBM2} and \cite{RBM1} indicates that the parametrix $ Q $ of $ -\Delta_{g} $ here is a bounded linear operator $ Q : \bar{H}^{s}(M, M \times \mathbb{C}) \rightarrow \bar{H}^{s +2}(M, M \times \mathbb{C}) $, and in particular with (\ref{twentyseven}), $ Q $ is a map between $ \mathcal{C}^{\infty}(M, M \times \mathbb{C}) $. Analogous in Euclidean case, we observe that
\begin{equation*}
Q(-\Delta_{g})u_{k} + QF(u_{k-1}) = Qf \Rightarrow u_{k} - Ru_{k} + QF(u_{k-1}) = Qf.
\end{equation*}
\noindent The regularity condition assures that we can pass through the limit in the sense of $ H_{0}^{1} $ - norm and hence it follows that
\begin{equation*}
u = Ru - QF(u) + Qf.
\end{equation*}
\noindent Apply $ \lVert \cdot \rVert_{H_{0}^{1}(M)} $ norm on both sides of (\ref{twentyeight}) as well as the usual argument with respect to globally Lipschitz function as in Theorem 3.4, we observe that
\begin{align*}
& \lVert u \rVert_{H_{0}^{1}(M)} \leqslant \lVert R \rVert_{H_{0}^{1}(M) \rightarrow H^{-N}(M)} \lVert u \rVert_{H^{-N}(M)} \\
& \qquad + \lVert Q \rVert_{L^{2}(M) \rightarrow H^{2}(M)} \lVert C \rVert \lVert u \rVert_{H_{0}^{1}(M)} + \lVert Q \rVert_{L^{2}(M) \rightarrow H^{2}(M)} \lVert f \rVert_{L^{2}(M)}.
\end{align*}
\noindent If the hypothesis in the statement holds, we then shift the middle term above to LHS and finally conclude that
\begin{equation*}
\lVert u \rVert_{H_{0}^{1}(M)} \leqslant C' \lVert f \rVert_{L^{2}(M)} + C_{N} \lVert u \rVert_{H^{-N}(M)}.
\end{equation*}
\noindent Note that $ C_{N} $ depends on the choice of Sobolev space $ H^{-N}(M) $.
\end{proof}
\noindent Again, (\ref{twentyeight}) is an integration version of the PDE (\ref{twentynine}) with further assumptions of regularity of $ f $ and the Riemannian metric $ g $ stated above. 

\section*{Acknowledgement}
The author would like to thank his advisor Prof. Steven Rosenberg for his support and mentorship.

\bibliographystyle{plain}
\bibliography{Elliptic}

\end{document}